\documentclass[english,reqno]{amsart}

\usepackage{amsmath, appendix, ulem}
\usepackage[nobysame]{amsrefs}
\usepackage{amssymb, color}
\usepackage[dvipsnames]{xcolor}
\usepackage[margin=1in]{geometry}

\usepackage{mathrsfs}
\usepackage{graphicx}
\usepackage{subfig}
\usepackage{float}
\usepackage{epsf}
\usepackage{xcolor}
\usepackage{hyperref}
\usepackage{titletoc}

\graphicspath{{../Figures/}}

\numberwithin{equation}{section}

\newtheorem{lem}{Lemma}[section]
\newtheorem{thm}{Theorem}[section]

\theoremstyle{remark}
\newtheorem{rmk}{Remark}[section]

\renewcommand{\tilde}{\widetilde}
\renewcommand{\hat}{\widehat}
\renewcommand{\bar}{\overline}

\newcommand{\nn}{\nonumber}
\newcommand{\R}{{\mathbb R}}

\newcommand{\del}{\partial}

\newcommand{\dt}{ \, {\rm d} t}
\newcommand{\dx}{ \, {\rm d} x}
\newcommand{\dy}{ \, {\rm d} y}
\newcommand{\dz}{ \, {\rm d} z}
\newcommand{\dv}{ \, {\rm d} v}
\newcommand{\dw}{ \, {\rm d} w}

\newcommand{\dmu}{\, {\rm d} \mu}

\newcommand{\One}{\boldsymbol{1}}

\newcommand{\Id}{{\rm{Id}}}
\newcommand{\Ss}{{\mathbb{S}}}

\newcommand{\Eps}{\epsilon}
\newcommand{\Ni}{\noindent}

\newcommand{\BigO}{{\mathcal{O}}}
\newcommand{\Supp}{supp \,}

\newcommand{\CalP}{{\mathcal{P}}}

\newcommand{\Null} {{\rm Null} \, }
\newcommand{\Span} {{\rm Span} \, }

\newcommand{\VV}{\mathbb{V}}

\newcommand{\abs}[1]{\left\lvert#1\right\rvert}
\newcommand{\norm}[1]{\left\lVert#1 \, \right\rVert}
\newcommand{\vint}[1]{\left\langle#1\right\rangle}

\newcommand{\vpran}[1]{\left(#1\right)}

\begin{document}

\title[Kinetic Equation with Internal States]{Multiple  asymptotics  of Kinetic Equations with Internal States}

\author{Benoit Perthame}
\address{Sorbonne Universit{\'e}, CNRS, Universit\'{e} de Paris, Inria, Laboratoire Jacques-Louis Lions UMR7598, F-75005 Paris}
\email{Benoit.Perthame@sorbonne-universite.fr}
\author{Weiran Sun}
\address{Department of Mathematics, Simon Fraser University, 8888 University Dr., Burnaby, BC V5A 1S6, Canada}
\email{weirans@sfu.ca}
\author{Min Tang}
\address{Institute of natural sciences and department of mathematics, 
Shanghai Jiao Tong University, Shanghai, 200240, China}
\email{tangmin@sjtu.edu.cn}
\author{Shugo Yasuda}
\address{Graduate School of Simulation Studies, University of Hyogo, 650-0047 Kobe, Japan}
\email{yasuda@sim.u-hyogo.ac.jp}
\date{\today}
\begin{abstract}
The run and tumble process is well established in order to describe the  movement of bacteria in response to a chemical stimulus. However the relation between the tumbling rate and the  internal state of bacteria is poorly understood. The present study aims at deriving models at the  macroscopic scale from assumptions  on the microscopic scales. In particular we are interested in comparisons between the stiffness of the response and the adaptation time. Depending on the asymptotics chosen both the standard Keller-Segel equation and the flux-limited Keller-Segel (FLKS) equation can appear. An interesting mathematical issue arises with a new type of equilibrium equation leading to solution with singularities. 
  
\end{abstract}


\maketitle
%
\section{Introduction}
Chemotaxis is the movement of bacteria in response to a chemical stimulus. 
 Individual bacteria moves by alternating forward-moving runs and reorienting tumbles. The velocity jump model has been proposed to describe the switching between these two states, in which the two processes, run-to-tumble and tumble-to-run, are modelled by two Poisson processes. The frequencies of these two poisson processes are determined by intracellular molecular biochemical pathway. Therefore, to study the chemotaxis behaviour quantitatively, it is crucial to understand the response of bacteria to signal changes and relate this information to the switching frequency. The mechanism has been well understood for {\it Escherichia coli ({\it E. coli})} chemotaxis and other bacteria using similar strategies to move have also been observed (\cite{Duffy97,Hazelbauer12}). In what follows we will focus on the behavior of {\it E. Coli}.

It is known that {\it  E. coli} responds to signal changes in two steps: excitation and adaptation. Excitation is when {\it E. coli} rapidly changes the tumbling frequency as it detects external signal changes, while the slow adaptation allows the cell to relax back to the basal tumbling frequency (the frequency when the intracellular chemical reactions are at equilibrium) \cite{JOT}. In the simplest description of the biochemical pathways, a single variable $m$ is used to represent the intracellular methylation level. The methylation has an equilibrium level $M(t, x)$ which is a function of extra-cellular chemical concentration.
Using $F(m,M)$ as the adaptation rate, the equation for the adaptation process has the form
\begin{align*}
\frac{{\rm d}m}{\dt}=F(m,M(t, x)) \,.
\end{align*}
The switching frequencies of run-to-tumble and tumble-to-run are determined by both $m$ and $M(t, x)$. The tumbling time can be ignored since it is usually much shorter than the running time, thus one can combine the two successive processes run-to-tumble and tumble-to-run together and assume that once the bacterium stops run it will immediately choose randomly some direction and start running again. Tumbling frequency is used to describe this Poisson process and it depends on the methylation level $m$ and its equilibrium $M(t, x)$.

In order to understand the relation between individual bacteria movement and their population level behaviour, 
pathway-based kinetic-transport model has been proposed in \cite{ErbanOthmer04,SWOT}. This model governs the evolution of the probability density function $p(t, x,v,m)$ of the bacteria at time $t$, position $x \in\R^d$, velocity $v \in \VV$, and methylation level $m$. The velocity space $\VV$ is $[-V_0,V_0]$ for $d=1$ and $V_0\mathbb{S}^d$ for $d\geq 2$ (the sphere with radius $V_0$). The general form of the kinetic equation is
\begin{align} \label{eq:kinetic-m} 
  \del_t p + v \cdot \nabla_x p + \del_m [F(m,M)p] 
  &= Q[m,M](p) \,,
\end{align}
where the tumbling term $Q[m,M](p)$ satisfies
\begin{align} \label{def:Q}
 Q[m,M](p)
 = \frac{1}{||\VV||}\int_\VV  \left[
     \lambda (m,M, v, v') \, p(t, x, v', m)  
     - \lambda (m,M, v', v) \, p(t, x, v, m) \right] \dv'.
\end{align}
Here $\lambda(m, M, v, v')$ is the methylation-dependent tumbling frequency from $v'$ to $v$ and $||\VV||=\int_\VV dv$.

The pathway-based kinetic equation in \eqref{eq:kinetic-m} is at the mesoscopic level. It can bridge the microscopic (individual) and macroscopic (population) level models by using moment closure or asymptotic analysis \cite{ErbanOth2,SWOT,ST,STY,Xue2}.
In this paper, we are interested in building such connections of some macroscopic models with \eqref{eq:kinetic-m}. Particular attention will be paid to the flux-limited Keller-Segel (FLKS) model. The most standard and popular macroscopic population level model describing the dynamics of the bacteria density is the classical Keller-Segel (KS) model \cite{Patlak,KS1,KS2}. There has been extensive mathematical studies of the KS equation \cite{BBW2015}. Analytically it has been discovered that depending on the cell number, solutions of KS model can either undergo smooth dispersion or blow up in finite time \cite{CW08,BCL09}. 
The reason that blow-up could happen is because the drift velocity is proportional to the gradient of the external chemicals $|\nabla S|$. Therefore it is not bounded when $|\nabla S|\to \infty$. This is in discrepancy with reality since the population level drift velocity is expected to saturate when the chemical gradient is large. Hence in \cite{HP09} a more physically relevant FLKS model is proposed, which has the form 
 \begin{align} \label{eq:FLKS}
   &\del_t \rho - \nabla_x \cdot \vpran{D\nabla_x \rho}
   + \nabla_{x} \vpran{\rho \, \phi(|\nabla_x S|) \nabla S} = 0, \\
   &\del_tS - D_S \Delta S+\alpha S=\rho,
\end{align} 
 where $\rho(t, x)$ is the cell density at time $t$ and position $x$, $\phi(|\nabla_x S|)|\nabla_x S|$ is a bounded function in $|\nabla_x S|$. Unlike the classical KS model, the solution to \eqref{eq:FLKS} exists globally in time \cite{Chertock,BW2017}. Interesting features of the FLKS are specific stiff response induced unstability and existence of traveling waves (stable or unstable), \cite{PeYa2018, CPY2018}.
 
The classical KS model has been recovered from the kinetic equation \eqref{eq:kinetic-m} as diffusion limits, as well as the FLKS \cite{PVW2018}. 
 In \cite{ErbanOthmer04,XO}, the authors derived the KS equation by incorporating the linear adaptation. Recently, intrinsically nonlinear signaling pathway are considered \cite{STY,Xue2}. The derivations in \cite{ErbanOthmer04, STY, SWOT,XO,Xue2} are formal and are based on moment closure techniques.
 When the internal state is not far from its expectation, the moment closure method provides the correct behavior of {\it E. coli} on the population level. However, the closeness assumption is only valid when the chemical gradient is small so that the internal states of different bacteria can concentrate and yield the KS equation \cite{SWOT,ST,Xue1}. In general this assumption does not always apply for other scales of the chemical gradient and adaptation time. 
Our goal is to derive different macroscopic models for different chemical gradient and adaptation time by asymptotic analysis.

We assume that the tumbling frequency $\lambda$ depends only on $M(t,x)-m$. This assumption is valid for {\it E. coli} chemotaxis \cite{SWOT}. Moreover, we use the following linear model for the adaptation 
\begin{equation}\label{eq:F}
\frac{{\rm d}m}{\dt}=F(m,M(t,x))=\frac{1}{\tau}\big(M(t,x)-m\big),
\end{equation}
where $\tau$ gives the characteristic time scale of the adaptation process \cite{ErbanOthmer04}.
Introducing a new variable 
$$
y=M(t,x)-m,
$$
and letting $f(t,x,v,y)=p(t,x,v,m)$,
one can rewrite the original model \eqref{eq:kinetic-m} as
\begin{align} \label{eq:kinetic-y} 
  \del_t f + v \cdot \nabla_x f +(\del_tM+v\cdot\nabla_xM)\del_yf-\frac{1}{\tau}\del_y (yf) 
  &= \lambda(y)\left[\frac{1}{||\VV||}\int_\VV f(t, x, v', y)\dv' - 
     f(t, x, v, y) \right]\,.
\end{align}
In this paper, we only consider the case when $M(t,x)$ is independent of time and $\nabla_xM$ is uniform in space. This special case is physically interesting: for {\it E. coli} chemotaxis, since $M(t,x)$ is related to the extra-cellular attractant profile by a logarithmic dependency, a uniform $\nabla_xM$ corresponds to the exponential environment as in the experiment in 
 \cite{Kalinin}. Let $G$ be the constant vector given by
\begin{equation} \label{eq:G}
    G = \nabla_x M \,.
\end{equation}
Let $\lambda_0$ be the characteristic rate of tumbling and $\delta$ the stiffness of the response. We denote 
\begin{align*}
   \vint f =\frac{1}{||\VV||}\int_\VV f(t,x,v',y)dv' \,,
\qquad
   \lambda(y)=\lambda_0\Lambda\left(\frac{y}{\delta}\right) \,.
\end{align*} 
Then equation~\eqref{eq:kinetic-y} becomes
\begin{align} \label{eq:kinetic-yG} 
  \del_t f + v \cdot \nabla_x f +\del_y \left( \left(v\cdot G-\frac{y}{\tau} \right)f\right) 
  &= \lambda_0\Lambda\left(\frac{y}{\delta}\right)\left(\vint f - 
     f \right)\,,
\end{align}
To perform the asymptotic analysis, we introduce the non-dimensional variables 
\begin{align*}
   \tilde x=x/L_0 \,,
\quad 
   \tilde v=v/V_0 \,,
\quad
   \tilde t=t/t_0 \,,
\quad 
  \tilde y=y/\delta \,,
\end{align*}
where $L_0$, $V_0$, and $t_0$ are the characteristic length, speed, and time respectively. Then the equation \ref{eq:kinetic-yG} is reformulated as 
\begin{align}\label{eq:kinetic-yG2}
  \sigma \del_{\tilde{t}} f + \tilde{v} \cdot \nabla_{\tilde x} f +\del_{\tilde y} \left(\left(\tilde{v}\cdot \frac{\tilde G}{\delta}-\frac{\tilde y}{\tilde \tau} \right)f\right) 
  &= {\tilde \lambda_0}\Lambda\left(\tilde y\right)\left(\vint f - 
     f \right)\,,
\end{align}
with the parameters given by 
\begin{align*}
\quad \tilde \lambda_0= \lambda_0 L_0/V_0 \,,
\quad
   \tilde \tau=\tau/(L_0/V_0) \,,
\quad
   \tilde G=G L_0 \,,
\quad
   \sigma=L_0/(t_0V_0) \,.
\end{align*}
Since there will be no confusion, in the remainder of the paper we drop the tilde sign for simplicity. The final form of the kinetic equation under consideration is 
\begin{align}\label{eq:kinetic-scaled}
  \sigma \del_t f + v \cdot \nabla_{x} f +\del_{y} \left(\left(v\cdot \frac{G}{\delta}-\frac{y}{\tau} \right)f\right) 
  &= {\lambda_0}\Lambda\left(y\right)\left(\vint f - 
     f \right)\,,
\end{align}

Our goal in this paper is to start from \eqref{eq:kinetic-scaled} and systematically derive macroscopic limits from it:  assuming $\lambda_0$ is large and $G=O(1)$, we consider the seamless reorientation by runs and tumbles with different scalings of $\tau$ and $\delta$.
More specifically, we investigate the following three cases:
\begin{itemize}
\item[I.] Fast adaptation and stiff response: both $\tau$ and $\delta$ are as small as $\lambda_0^{-1}$;
\item[II.] Very fast adaptation and very stiff response: both $\tau$ and $\delta$ are much smaller than $\lambda_0^{-1}$;
\item[III.] Moderate  adaptation and moderate response, i.e. both $\tau$ and $\delta$ are of $O(1)$.
\end{itemize}
The common theme in these limits is to decide when the classical KS and the FLKS equations will occur. A brief summary of our result is
\begin{enumerate}
\item 
In Case I, both the hyperbolic model and a FLKS type model can be found. In particular, when the leading order behaviour of the bacteria population satisfies a hyperbolic model, we further consider the motion of its front profile and find a  FLKS model in the moving frame.
\item 
In Case II,  a FLKS type model can be derived.
\item
 In Case  III the solution will tend to the solution of a classical KS model.
\end{enumerate}

The rest of the paper is organized as follows.
In Section 2, we use asymptotic analysis to formally derive the leading-order macroscopic equations from~\eqref{eq:kinetic-m}.
Some properties of the leading order distribution are rigorously shown in Section 3. In Section 4, we show numerical results that are consistent with the analytical properties in Section 3. Finally we conclude in Section 5.

\section{Formal Asymptotic Limits}
In this section we show that both classical and flux-limited Keller-Segel equations can arise in different physical regimes characterized by the stiffness of the chemotactic response and the rate of adaptation. 
In all the cases considered below we fix $\lambda_0$ as
\begin{align*}
   \lambda_0 = \frac{1}{\Eps}
\qquad \text{with} \qquad
   \Eps \ll 1 \,.
\end{align*}


\subsection{Fast adaptation and stiff response} 
In this case we assume that
\begin{align*}
\delta=\tau=\epsilon = \lambda_0^{-1} \,,
\qquad
\sigma = \Eps^{\alpha-1} \,,
\qquad
\alpha = 1 \,\, \text{or} \,\, 2 \,. 
\end{align*}
We also assume that the tumbling frequency $\Lambda$ has the specific structure
\begin{align} \label{assump:Lambda-FLKS}
   \Lambda (y) = \Lambda_0(y) + \Eps \Lambda_1(y) > 0 \,,
\qquad
   \Lambda_0, \, \Lambda_1 \in C_b(\R) \,,
\end{align}
where $\Lambda_0$ is a strictly positive continuous function and $\Lambda_1$ is independent of $\Eps$. The space $C_b(\R)$ is the collection of continuous and bounded functions on $\R$. 

Consider the scaled equation
\begin{align} \label{eq:q-Eps-alpha}
   \Eps^\alpha \del_t q_\Eps + \Eps v \cdot \nabla_x q_\Eps
  + \del_y \vpran{(v \cdot G - y) q_\Eps} =  \Lambda(y) (\vint{q_\Eps} - q_\Eps) \,,
\qquad
  (y, v) \in \R \times \VV \,.
\end{align}
The leading-order term of $q_\Eps$ is of the form $\rho (t,x) Q_0(y,v)$ where $Q_0$ satisfies 
\begin{align} \label{def:Q-0-shift}
   \del_y \vpran{(v \cdot G - y) Q_0} =  \Lambda_0(y) (\vint{Q_0} - Q_0), \qquad Q_0 \geq 0 , \quad \iint_{\R \times \VV} Q_0 \dy \dv =1 \, .
\end{align}
The existence and uniqueness of such $Q_0$ is shown in Theorem~\ref{thm:Q-0-general}. 

\begin{rmk}
In Theorem~\ref{thm:Q-0-general} we show that $Q_0$ is compactly supported on $[-|G|, |G|] \times \VV$ and is strictly positive in $(-|G|, |G|) \times \VV$. In the rest of the current section, all the integration involving $Q_0$ are performed over $(-|G|, |G|) \times \VV$,  although we often write the integration domain as $\R \times \VV$ for the sake of simplicity in notation. 
\end{rmk}

Let $v_0$ be the leading-order average velocity given by
\begin{align} \label{def:v-0}
   v_0 = \int_\R \int_\VV v \, Q_0 \dv \dy \,.
\end{align}
In the case where $v_0 \neq 0$, by letting $\alpha = 1$ one can derive that the leading-order approximate of $q_\Eps$ satisfies a pure transport or hyperbolic equation. More specifically, we have the formal limit
\begin{thm} \label{thm:hyperbolic}
Suppose $\alpha = 1$ and $\Lambda$ satisfies~\eqref{assump:Lambda-FLKS} together with the assumptions in Theorem~\ref{thm:Q-0-general}. Suppose the average velocity $v_0$ defined in~\eqref{def:v-0} is nonzero. Let $q_\Eps$ be the non-negative solution to~\eqref{eq:q-Eps-alpha}. Suppose $q_\Eps \to q_0$ in $L^1(\!\dv\dy\dx)$. Then $q_0 = \rho_0(t, x) Q_0(y, v)$ and $\rho_0$ satisfies the hyperbolic equation
\begin{align} \label{limit:hyperbolic}
   \del_t \rho_0 + \nabla_x \cdot \vpran{v_0 \rho_0} = 0 \,.
\end{align}
\end{thm}
\begin{proof}
Since~\eqref{eq:q-Eps-alpha} is a linear equation, we can pass $\Eps \to 0$ and obtain $q_0$ as a non-negative $L^1$ solution to the steady-state equation such that
\begin{align*}
    \del_y \vpran{(v \cdot G - y) q_0} =  \Lambda_0(y) (\vint{q_0} - q_0) \,.
\end{align*}
By the uniqueness shown in Theorem~\ref{thm:Q-0-general}, $q_0$ is a multiple of $Q_0$ with a coefficient $\rho_0 \in L^1(\!\dx)$. The conservation law associated with~\eqref{eq:q-Eps-alpha} and $\alpha=1$ has the form
\begin{align*}
   \del_t \rho_\Eps + \nabla_x \cdot \int_\R \int_\VV v q_\Eps \dv \dy = 0 \,,
\end{align*}
The limiting equation~\eqref{limit:hyperbolic} is obtained by passing $\Eps \to 0$ and using that $q_0 = \rho_0 Q_0$.
\end{proof}

To observe nontrivial diffusive behaviour from equation~\eqref{eq:q-Eps-alpha}, we make a change of variable such that
\begin{align*}
   f_\Eps(t, x, y, v)= q_\Eps(t, x + v_0 t, y, v) \,,
\end{align*}
where again $v_0$ is the average speed defined in~\eqref{def:v-0}. Physically speaking, we consider the motion of the front profile of $q_\Eps$.
Now we let $\alpha = 2$. Then the equation for $f_\Eps$ reads
\begin{align} \label{eq:f-Eps-shift}
   \Eps^2 \del_t f_\Eps + \Eps (v - v_0) \cdot \nabla_x f_\Eps
  + \del_y \vpran{(v \cdot G - y) f_\Eps} =  \Lambda(y) (\vint{f_\Eps} - f_\Eps) \,.
\end{align}
Before showing the formal limit for $f_\Eps$, we prove a lemma of the classical entropy-estimate type:
\begin{lem} \label{lem:positivity-entropy}
Let $T_0$ be the operator defined as 
\begin{align} \label{def:T}
  T_0g = \del_y \vpran{(v \cdot G - y) g} -  \Lambda_0(y) (\vint{g} - g) \,.
\end{align}
Let $Q_0$ be the unique solution in~\eqref{def:Q-0-shift}. Then for any $g$ that makes sense of all the integrals and satisfies that 
\begin{align} \label{cond:vanishing-IBP}
    \lim_{y \to \pm |G|} \frac{g^2}{Q_0}(v, y) = 0
\qquad
   \text{for all $v \in \VV$.}
\end{align}
we have
\begin{align*}
   \int_\R \int_\VV \frac{g}{Q_0} T_0g \dv\dy 
\geq 
   \frac{1}{2} \int_\R \int_\VV  \Lambda_0(y) 
       \vpran{\frac{g}{Q_0} - \vint{\frac{g}{Q_0}}}^2 Q_0(v) \dv\dy
\geq 0 \,.
\end{align*}
Note that a particular case for~\eqref{cond:vanishing-IBP} to hold is when $g/Q_0 \in L^\infty(\R \times \VV)$. 
\end{lem}
\begin{proof}
The proof follows from a direct calculation. Indeed, by integration by parts,
\begin{align*}
   \int_\R \int_\VV \frac{g}{Q_0} T_0g \dv\dy
&= \int_\R \int_\VV \frac{g}{Q_0}
   \del_y \vpran{(v \cdot G - y) \frac{g}{Q_0} Q_0} \dv\dy
   - \int_\R \int_\VV \frac{g}{Q_0} \Lambda_0(y)(\vint{g} - g) \dv\dy
\\
& = \frac{1}{2} \int_\R \int_\VV
      \vpran{\frac{g}{Q_0}}^2 \del_y \vpran{(v\cdot G - y) Q_0} \dv\dy
      - \int_\R \int_\VV \frac{g}{Q_0} \Lambda_0(y) (\vint{g} - g) \dv\dy \,,
\end{align*}
where condition~\eqref{cond:vanishing-IBP} is used to guarantee that the boundary terms vanish in the integration by parts.
Using the equation for $Q_0$, we have
\begin{align*}
    \int_\R \int_\VV \frac{g}{Q_0} T_0g \dv\dy
& = \frac{1}{2} \int_\R \int_\VV
      \vpran{\frac{g}{Q_0}}^2 \Lambda_0(y) \vpran{\vint{Q_0} - Q_0} \dv\dy
      - \int_\R \int_\VV \frac{g}{Q_0} \Lambda_0(y) (\vint{g} - g) \dv\dy \,.
\\
& = \frac{1}{2} \int_\R \int_\VV \Lambda_0(y)
      \vpran{\vpran{\frac{g}{Q_0}}^2 \vint{Q_0}
      - 2\frac{g}{Q_0} \vint{g}
      + \frac{g^2}{Q_0}} \dv\dy
\\
& \geq \frac{1}{2} \int_\R \int_\VV  \Lambda_0(y) 
       \vpran{\frac{g}{Q_0} - \vint{\frac{g}{Q_0}}}^2 Q_0(v) \dv\dy \geq 0 \,.
\end{align*}
The equal sign holds only when $g/Q_0$ is independent of $v$. \end{proof}

Equipped with Lemma~\ref{lem:positivity-entropy}, we show the formal asymptotic limit of~\eqref{eq:f-Eps-shift} in the following theorem:
\begin{thm} \label{thm:front-asymp}
Let $f_\Eps$ be the non-negative solution to~\eqref{eq:f-Eps-shift} and $\rho_\Eps = \int_\R \int_\VV f_\Eps(t,x, y, v) \dv\dy$ such that
\begin{align*}
   f_\Eps \to f_0 \quad \text{locally  in $L^1(\!\dt \dv\dy\dx)$}
\end{align*}
for some $f_0(t)  \in L^1(\dv\dy\dx)$. Then $f_0 = \rho_0(t, x) Q_0(y,v)$ and $\rho_0$ satisfies the drift-diffusion equation
\begin{align} \label{eq:FLKS1}
   \del_t \rho_0 - \nabla_x \cdot \vpran{D_{0, 2} \nabla_x \rho_0}
   + \nabla_{x} \vpran{c_{0, 2} \rho_0} = 0 \,,
\end{align}
where the diffusion coefficient and drift velocity are given by 
\begin{align} 
   D_{0, 2} &=  \int_\R \int_\VV (v-v_0)  \otimes T_0^{-1}((v-v_0) Q_0) \dv \dy  \,, \label{def:D-0-2}
\\
   c_{0,2} &= \int_\R \int_{\VV} (v - v_0) T^{-1}_0 \vpran{\Lambda_1(y) (\vint{Q_0} - Q_0)} \dv\dy \,. \label{def:c-0-2}
\end{align}
Moreover, $D_{0, 2}$ is strictly positive. 
\end{thm}

We conjecture that $c_{0,2} $ is bounded in $G$, thus rising again a FLKS type equation. A possible route toward a proof is as follows. We write
\[
c_{0,2} = \int_\R \int_{\VV} T^{-1\ast}_0 (v - v_0)  \;  \Lambda_1(y) (\vint{Q_0} - Q_0) \dv\dy
\]
and it remains to establish that $T^{-1\ast}_0 (v - v_0)$ is uniformly bounded in $G$ in $L^\infty$. 
Numerics sustain this boundedness, see Fig.~\ref{fig_c0}.

\begin{proof}
First, we use the same argument as in Theorem~\ref{thm:hyperbolic} to deduce that there exists $\rho_0 \in L^1(\!\dx)$ such that $f_0 = \rho_0 Q_0$ and $\rho_\Eps \to \rho_0$ in $L^1(\!\dx)$. Next, the conservation law for~\eqref{eq:f-Eps-shift} reads
\begin{align} \label{conserv:diffusion-front}
   \del_t \int_\R \int_\VV f_\Eps \dv\dy
+ \frac{1}{\Eps} \nabla_x \cdot \int_\R \int_\VV
   (v - v_0) f_\Eps \dv \dy 
= 0 \,.
\end{align}
By the definition of $v_0$ and $Q_0$, we have
\begin{align*}
   \int_\R \int_\VV (v-v_0) Q_0(y, v) \dv\dy = 0
\end{align*}
Hence ~\eqref{conserv:diffusion-front} can also be written as
\begin{align} \label{conserv:diffusion-front-revise}
   \del_t \int_\R \int_\VV f_\Eps \dv\dy
+ \frac{1}{\Eps} \nabla_x \cdot \int_\R \int_\VV
   (v - v_0) \vpran{f_\Eps - \rho_\Eps Q_0} \dv \dy 
= 0 \,.
\end{align}
Then we can re-write equation~\eqref{eq:f-Eps-shift} as
\begin{align*}
  T_0 \vpran{f_\Eps - \rho_\Eps Q_0} 
= -\Eps \vpran{\nabla_x \rho_\Eps} \cdot (v - v_0) Q_0
   - \Eps (v - v_0) \cdot \nabla_x \vpran{f_\Eps - \rho_\Eps Q_0}
   - \Eps^2 \del_t f_\Eps
   + \Eps \Lambda_1(y) \vpran{\vint{f_\Eps} - f_\Eps} \,.
\end{align*}
By Theorem~\ref{thm:Fredholm}, we apply the pseudo-inverse of $T_0$ and obtain that 
\begin{align*} 
   \frac{1}{\Eps} \vpran{f_\Eps - \rho_\Eps Q_0}
&= -\vpran{\nabla_x \rho_\Eps} \cdot T_0^{-1}\vpran{(v-v_0) Q_0}
   + T_0^{-1} \vpran{\Lambda_1(y) (\vint{f_\Eps} - f_\Eps)} 
\\
& \quad \,
   - T_0^{-1} \vpran{(v - v_0) \cdot \nabla_x \vpran{f_\Eps - \rho_\Eps Q_0}
   + \Eps \del_t f_\Eps} \,.
\end{align*}
Note that every term on the right-hand side of the above equality is well-defined since the terms inside $T_0^{-1}$ all satisfy the orthogonality condition given in~\eqref{cond:R-1}. By the limit $f_\Eps - \rho_\Eps Q_0 \to 0$ we have
\begin{align*}
   (v - v_0) \cdot \nabla_x \vpran{f_\Eps - \rho_\Eps Q_0}
   + \Eps \del_t f_\Eps \to 0 
\qquad
  \text{in the sense of distributions.}
\end{align*}
Hence, it formally holds that
\begin{align*}
   \frac{1}{\Eps} \vpran{f_\Eps - \rho_\Eps Q_0}
\to -\vpran{\nabla_x \rho_0} \cdot T_0^{-1}\vpran{(v-v_0) Q_0}
     + \rho_0 T^{-1}_0 \vpran{\Lambda_1(y) (\vint{Q_0}) - Q_0}
\qquad
  \text{as $\Eps \to 0$.}
\end{align*}
Inserting such limit in the limiting form of the conservation law in ~\eqref{conserv:diffusion-front} gives~\eqref{eq:FLKS} with $D_{0,2}$ and $c_{0,2}$ satisfying~\eqref{def:D-0-2} and~\eqref{def:c-0-2} respectively.

Finally, we show that $D_{0, 2}$ is a positive definite matrix. To this end, denote 
\begin{align*}
    h = T_0^{-1}((v-v_0) Q_0) \,,
\end{align*}
or equivalently, 
\begin{align} \label{eq:h}
 T_0h =  \del_y \vpran{(v \cdot G - y) h} -  \Lambda_0(y) (\vint{h} - h) 
= (v - v_0) Q_0 \,.
\end{align}
Then for any arbitrary $\alpha \in \R^d$, 
\begin{align*}
   T_0(h \cdot \alpha) =  \del_y \vpran{(v \cdot G - y) (h \cdot \alpha)} - \Lambda_0(y) (\vint{h \cdot \alpha} - (h \cdot \alpha)) 
= \vpran{(v - v_0) \cdot \alpha} Q_0 \,.
\end{align*}
Hence,
\begin{align*}
  \alpha^T D_{0, 2} \alpha 
=  \int_\R \int_\VV \vpran{(v-v_0) \cdot \alpha} T_0^{-1}\vpran{((v-v_0) \cdot \alpha) Q_0} \dv \dy
= \int_\R \int_\VV T_0(h \cdot \alpha) \frac{h \cdot \alpha}{Q_0}  \dv\dy \,.
\end{align*}
By Lemma~\ref{lem:positivity-entropy}, if we multiply $(h \cdot \alpha)/Q_0$ to $T_0(h \cdot \alpha)$ and integrate in $(y, v)$, then
\begin{align*}
  \int_\R \int_\VV T_0(h \cdot \alpha)\;  \frac{h \cdot \alpha}{Q_0}  \dv\dy
\geq 
  \frac{1}{2} \int_\R \int_\VV  \Lambda_0(y) 
       \vpran{\frac{h \cdot \alpha}{Q_0} - \vint{\frac{h \cdot \alpha}{Q_0}}}^2 Q_0(v) \dv\dy
\geq 0 \,.
\end{align*} 
This shows $\alpha^T D_{0,2} \alpha$ is non-negative and the equal sign of the above inequality holds only when $(h \cdot \alpha)/Q_0$ is independent of $v$. Now we show that for any $\alpha$ independent of $(y, v)$, the quantity $(h \cdot \alpha)/Q_0$ must depend on $v$. Hence the strict positivity of $\alpha^T D_{0,2} \alpha$ must hold. To this end, suppose on the contrary that there exists $\alpha_0 \in \R^d$ independent of $(y, v)$ and $\beta(y)$ such that $h \cdot \alpha_0 = \beta(y) Q_0$. Then
\begin{align*}
    \del_y \vpran{(v \cdot G - y) \beta(y) Q_0} 
- \Lambda_0(y) \beta(y) \vpran{\vint{Q_0} - Q_0}
= \vpran{(v - v_0) \cdot \alpha_0} Q_0  \,,
\end{align*}
which, by the definition of $Q_0$ simplifies to
\begin{align*}
   \beta'(y) (v \cdot G - y) = (v - v_0) \cdot \alpha_0 \,.
\end{align*}
Since $\alpha_0$ is independent of $y, v$, the only solution to such equation is $\beta' = 0$ and $\alpha_0 = 0$. Hence $\alpha^T D_{0,2} \alpha > 0$ for any $\alpha \neq 0$.
This proves the strict positivity of the diffusive coefficient matrix $D_{0, 2}$. 
\end{proof}

In order to justify that~\eqref{eq:FLKS} is indeed a flux-limited Keller-Segel equation, ideally we want to show that the drift velocity $c_{0,2}$ is generally nonzero, it is uniformly bounded in the chemical gradient $G$ and vanishes when $G$ approaches zero. In what follows we show that these properties are satisfied in the special case where $\Lambda_0$ is a constant. In Section~\ref{sec:numerics} we give some numerical evidence that $c_{0, 2}$ is uniformly bounded in $G$ even when $\Lambda_0$ depends on $y$.

\begin{lem} \label{lem:c02}
Suppose $\Lambda_0$ is a constant. Then the drift velocity $c_{0,2}$ satisfy that

\Ni (a) $c_{0,2}$ is nonzero. 

\Ni (a) $c_{0, 2}$ is bounded for all $G$. Moreover,  $c_{0,2} \to 0$ as $G \to 0$.

\end{lem}
\begin{proof}
\Ni (a) 
Since $\Lambda_0$ is a constant, $Q_0$ must satisfy the symmetry condition~\eqref{symm-Q-0} in Theorem~\ref{thm:Q-0-general}. Hence $v_0 = 0$. Moreover, we have the relation $\vpran{T^{-1}_0}^\ast v = \frac{1}{\Lambda_0} v$ since $v$ satisfies the equation
\begin{align*}
   (v \cdot G - y) \del_y v - \Lambda_0 (\vint{v} - v) = \Lambda_0 v \,.
\end{align*}
The drift velocity $c_{0, 2}$ is now reduced to 
\begin{equation} \label{eq_c02}
\begin{split}
   c_{0, 2}
&= \int_\R \int_{\VV} v \, T^{-1}_0 \vpran{\Lambda_1(y) (\vint{Q_0} - Q_0)} \dv\dy
\\
& = \int_\R \int_{\VV} \vpran{\vpran{T^{-1}_0}^\ast v} \vpran{\Lambda_1(y) (\vint{Q_0} - Q_0)} \dv\dy
\\
& = \frac{1}{\Lambda_0}\int_\R \int_{\VV} v \vpran{\Lambda_1(y) (\vint{Q_0} - Q_0)} \dv\dy
= - \frac{1}{\Lambda_0}\int_\R \int_{\VV} \Lambda_1(y) v Q_0 \dv\dy \,,
\end{split}
\end{equation}
We want to show that $\int_\VV v Q_0(y, v) \dv$ as a function in $y$ is not zero. To this end, we integrate~\eqref{def:Q-0-shift} in $v$. This gives
\begin{align*}
   \del_y \vpran{G \cdot \int_\VV v Q_0 \dv - y \int_\VV Q_0 \dv} = 0 \,.
\end{align*}
Since $Q_0$ is compactly supported on $[-G, G] \times \VV$, one must have
\begin{align*}
   G \cdot \int_\VV v Q_0 \dv - y \int_\VV Q_0 \dv = 0 \,.
\end{align*}
Since $\int_\VV Q_0 \dv$ is not a zero function in $y$, we deduce that $\int_\VV v Q_0 (y, v) \dv$ is not a zero function in $y$ either. Since $\int_{\VV}  v Q_0 \dv \in L^1(\!\dy)$, there exists $\Lambda_1 \in L^\infty(\!\dy)$ such that
\begin{align*}
   c_{0, 2} 
= - \frac{1}{\Lambda_0}\int_\R \Lambda_1(y) \vpran{\int_{\VV}  v Q_0 \dv}\dy
\neq 0 \,.
\end{align*}


\Ni (b) With the assumption that $\Lambda_0$ is a constant, we have derived the simplified form of $c_{0, 2}$ in~\eqref{eq_c02}. Then
\begin{align*}
   \abs{c_{0, 2}}
\leq
  \frac{\norm{\Lambda_1}_{L^\infty(\!\dy)}}{\Lambda_0}
  \iint_{\R \times \VV} Q_0 \dv \dy 
=  \frac{\norm{\Lambda_1}_{L^\infty(\!\dy)}}{\Lambda_0} \,.
\end{align*}
Hence $c_{0, 2}$ is uniformly bounded in $G$. Since $\Lambda_1$ is continuous and $\Supp Q_0 = [-|G|, |G|] \times \VV$, we have
\begin{align*}
   \lim_{G \to 0} c_{0, 2}
= -\frac{\Lambda_1(0)}{\Lambda_0}
   \iint_{\R \times \VV} v Q_0 \dy \dv = 0 \,.
\end{align*}
Hence $c_{0, 2}$ vanishes as $G$ approaches zero. 
\end{proof}

\subsection{Very fast adaptation and very stiff response}

In this subsection we show a second scaling where flux-limited Keller-Segel equations can also arise. We consider the regime that combines  the scalings in \cite{PTV2016} and \cite{PVW2018}. In particular, let
\begin{align*}
\delta=\tau=\epsilon^2 \,, \qquad \sigma = \Eps \,.
\end{align*}
Then the scaled equation becomes
\begin{align} \label{eq:f-Eps-fa}
   \Eps^2 \del_t f_\Eps + \Eps v \cdot \nabla_x f_\Eps
+ \frac{1}{\Eps} \del_y \vpran{(v \cdot G - y) f_\Eps} 
= \Lambda(y) \vpran{\vint{f_\Eps} - f_\Eps} \,.
\end{align}
Assume that $\Lambda$ satisfies the same condition as in~\cite{PVW2018} such that
\begin{align}\label{assump:barlambda0}
   \Lambda(y) = \bar\Lambda_0 + \Eps \Lambda_1(y) > 0 \,,
\end{align}
where $\bar\Lambda_0$ is a positive constant and $\Lambda_1 \in C_b(\R)$. Denote $\CalP(\R^d \times \R \times \VV)$ as the space of probability measures in $(x, y, v)$. The formal asymptotic limit is
\begin{thm} \label{thm:fast-adap-gradient}
Let $f_\Eps$ be the solution to~\eqref{eq:f-Eps-fa} with $\Lambda$ satisfying~\eqref{assump:barlambda0}. Suppose there exists $f_0(t, \cdot, \cdot, \cdot) \in \CalP(\R^d \times \R \times \VV)$ such that 
$f_\Eps \stackrel{\ast}{\rightharpoonup} f_0$ as probability measures. Then 
\begin{align*}
   f_0(t, x, y ,v) = \rho_0(t, x) \, \delta(y - v \cdot G)  \,.
\end{align*}
Moreover, $\rho_0$ satisfies the flux-limited Keller-Segel equation
\begin{align} \label{eq:FLKS-2}
   \del_t \rho_0 - \nabla_x \cdot \vpran{D_{0, 3} \nabla_x \rho_0}
   + \nabla_x \cdot \vpran{c_{0, 3} \rho_0} = 0
\end{align}
with
\begin{align} \label{def:D-3-c-3}
   D_{0, 3} = \frac{1}{\bar\Lambda_0} \int_{\VV} v_d^2 \dv > 0 \,,
\qquad
   c_{0, 3} = \frac{1}{\bar\Lambda_0}\int_\VV v \Lambda_1(v \cdot G) \dv \,.
\end{align}
\end{thm}
\begin{proof}
Passing $\Eps \to 0$ in~\eqref{eq:f-Eps-fa} gives the equation for $f_0$ as
\begin{align*}
   \del_y \vpran{(v \cdot G - y) f_0} = 0 \,.
\end{align*}
By the non-negativity of $f_0$ (and a similar argument as in Theorem~\ref{thm:Q-0-general} showing its support), we have
\begin{align*}
   (v \cdot G - y) f_0 = 0 \,.
\end{align*}
Hence, there exists $\rho_0(t, x) \geq 0$ such that 
\begin{align} \label{form:f-0-fa}
    f_0(t, x, y, v) = \rho_0(t, x) \delta(y - v \cdot G) \,.
\end{align}
The conservation law associated with~\eqref{eq:f-Eps-fa} is
\begin{align*}
   \del_t \rho_\Eps 
+ \frac{1}{\Eps} \nabla_x \cdot 
    \iint_{\R \times \VV} v f_\Eps \dy \dv = 0 \,.
\end{align*}
To derive the limit of the $\frac{1}{\Eps}$-term, we multiply~\eqref{eq:f-Eps-fa} by $v$ and integrate in $(y, v)$. This gives
\begin{align*}
  \Eps \del_t \int_\R \int_\VV v f_\Eps \dv \dy
+ \nabla_x \cdot \int_\R \int_\VV v \otimes v f_\Eps \dv\dy
= - \frac{\bar\Lambda_0}{\Eps} \int_\R \int_\VV v f_\Eps \dv\dy
   - \int_\R \int_\VV v \Lambda_1(y) f_\Eps \dv\dy \,.
\end{align*}
Passing $\Eps \to 0$ and applying~\eqref{form:f-0-fa}, we have
\begin{align*}
   \frac{1}{\Eps} \int_\R \int_\VV v f_\Eps \dv\dy
&\to 
   -\nabla_x \rho_0 \cdot \frac{1}{\bar\Lambda_0}\int_\VV v \otimes v  \dv
   - \vpran{\frac{1}{\bar\Lambda_0} \int_\VV v \Lambda_1(v \cdot G) \dv} \rho_0 
\\
& = -D_{0, 3} \nabla_x \rho_0 + c_{0, 3} \rho_0 \,,
\end{align*}
where $D_{0,3}, c_{0,3}$ satisfy~\eqref{def:D-3-c-3}.
\end{proof}

Note that in general the drift velocity $c_{0, 3}$ is non-zero. If $\VV = \Ss^{d}$ with $d \geq 1$, then by the rotational invariance, $c_{0,3}$ is along the direction of the chemical gradient $G$. Moreover, it is bounded for all $G$ as long as $\Lambda_1$ is a bounded function. Similar as in the proof of Lemma~\ref{lem:c02}, for $\Lambda_1$ being continuous, we have $c_{0, 3} \to 0$ as $G \to 0$. These observations justify that~\eqref{eq:FLKS-2} is of the flux-limited Keller-Segel type.


\subsection{Moderate Adaptation and Moderate Response} 
\label{sec:sasg}

The classical Keller-Segel equation can also be derived from kinetic equations with the internal state. The particular scaling is for a slow adaption and a moderate gradient where
\begin{align*}
  \delta=\tau=1 \,,
\qquad
   \sigma = \Eps \,.
\end{align*} The scaled equation is
\begin{align} \label{eq:f-Eps-KS}
  \Eps^2 \del_t f_\Eps + \Eps v \cdot \nabla_x f_\Eps
+ \Eps\del_y \vpran{(v \cdot G - y) f_\Eps}
&= \Lambda(y) \vpran{\vint{f_\Eps} - f_\Eps} \,.
\end{align}
In this case we only need to require that $\Lambda \in C^1(\R)$, in particular, it does not have to satisfy the specific form in~\eqref{assump:Lambda-FLKS}. The formal asymptotic limit is
\begin{thm}
Suppose $\Lambda \in C^1(\R)$ and $\Lambda(y) > 0$ for all $y \in \R$. 
Let $f_\Eps$ be the non-negative solution to~\eqref{eq:f-Eps-KS}. Suppose there exists $f_0(t, \cdot, \cdot, \cdot) \in \CalP(\R^d \times \R \times \VV)$ such that 
$f_\Eps \stackrel{\ast}{\rightharpoonup} f_0$ as probability measures. Then $f_0 = \rho_0(t, x) \delta_0(y)$ and $\rho_0$ satisfies the Keller-Segel equation
\begin{align*}
   \del_t \rho_0 - \nabla_x \cdot \vpran{D_{0, 4} \nabla_x \rho_0}
   + \nabla_x \cdot \vpran{c_{0, 4} \rho_0} = 0
\end{align*}
with the diffusion coefficient and drift velocity given by 
\begin{align*}
   D_{0, 4} = \frac{1}{\Lambda(0)} \int_{\VV} v_d^2 \dv > 0 \,,
\qquad
   c_{0, 4} = - G \frac{\Lambda'(0)}{\Lambda^2(0)}  \,.
\end{align*}
\end{thm}
\begin{proof}
The conservation law associated with~\eqref{eq:f-Eps-KS} is
\begin{align} \label{conserv-law-KS}
   \del_t \rho_\Eps 
+ \frac{1}{\Eps} \nabla_x \cdot 
    \iint_{\R \times \VV} v f_\Eps \dy \dv = 0 \,.
\end{align}
We apply the Hilbert expansion for equation~\eqref{eq:f-Eps-KS}. Write formally
\begin{align*}
   f_\Eps = f_0 + \Eps f_1 + \BigO(\Eps^2) \,,
\qquad
   f_0 \geq 0 \,.
\end{align*}
Then the leading order $f_0$ satisfies
\begin{align*}
    f_0 = \vint{f_0} \,.
\end{align*}
Hence $f_0$ is independent of $v$. The next order in equation~\eqref{eq:f-Eps-KS} gives
\begin{align*}
   \Lambda(y) \vpran{\vint{f_1} - f_1}
= v \cdot \nabla_x f_0 + \del_y \vpran{(v \cdot G - y) f_0} \,.
\end{align*}
The solvability condition requires that
\begin{align*}
   \int_\VV \vpran{v \cdot \nabla_x f_0 + \del_y \vpran{(v \cdot G - y) f_0}} \dv = 0  \,.
\end{align*}
Since $f_0$ is independent of $v$, this reduces to
\begin{align*}
   \del_y \vpran{y \int_\VV f_0 \dv} = 0 
\end{align*}
By the non-negativity of $f_0$, we must have
\begin{align*}
    y f_0 = \frac{1}{\norm{\VV}}\int_\VV y f_0 \dv = 0 
\qquad
    \text{for all $y \in \R$.}
\end{align*}
Hence, $f_0$ concentrates at $y=0$, that is, there exists $\rho_0 \geq 0$ such that
\begin{align*}
    f_0(t, x, y, v) = \rho_0(t, x) \delta_0(y) \,.
\end{align*}
Divide~\eqref{eq:f-Eps-KS} by $\Lambda$, and then multiply by $v$ and integrate in $(y, v)$. We obtain
\begin{align*}
   - \frac{1}{\Eps} \int_\R \int_\VV v f_\Eps \dv\dy
&= \Eps \del_t \int_\R \int_\VV \frac{v}{\Lambda(y)} f_\Eps \dv\dy
+ \nabla_x \cdot \int_\R \int_\VV v \otimes v \frac{1}{\Lambda(y)}f_\Eps \dv \dy
\\
& \quad \, 
  + \int_\R \int_\VV \frac{v}{\Lambda(y)} \del_y \vpran{(v \cdot G - y) f_\Eps} \dv \dy \,.
\end{align*}
Passing $\Eps \to 0$, we formally obtain that 
\begin{align*}
  - \lim_{\Eps \to 0 } \vpran{\frac{1}{\Eps} \int_\R \int_\VV v f_\Eps \dv\dy}
= \vpran{\frac{1}{\Lambda(0)} \int_{\VV} v_d^2 \dv} \nabla_x \rho_0
+ \vpran{G \frac{\Lambda'(0)}{\Lambda^2(0)}\int_{\VV}v_d^2\dv }  \rho_0 \,,
\end{align*}
Apply such limit in~\eqref{conserv-law-KS} gives the regular Keller-Segel where the drift velocity $c_{0, 4}$ is linear in the chemical gradient $G$ as long as $\Lambda'(0) \neq 0$. 
\end{proof}

\begin{rmk}
One can consider a fourth case with slow adaptation and stiff response where
\begin{align*}
      \delta=\Eps \,,
\qquad
       \tau=1 \,,
\qquad
   \sigma = \Eps \,.
\end{align*}
The scaled equation is
\begin{align} \label{eq:f-Eps-SA}
  \Eps^2 \del_t f_\Eps + \Eps v \cdot \nabla_x f_\Eps
+ \del_y \vpran{(v \cdot G -  \Eps y) f_\Eps}
&= \Lambda (y) \vpran{\vint{f_\Eps} - f_\Eps} \,.
\end{align}
If we change the variable $y $ to $z= \Eps y$, then the model becomes 
\begin{align} \label{eq:f-Eps-SAz}
  \Eps^2 \del_t f_\Eps +  \Eps v \cdot \nabla_x f_\Eps
+ \Eps \del_z \vpran{(v \cdot G   -  z) f_\Eps}
&= \Lambda\vpran{\frac z \Eps} \vpran{\vint{f_\Eps} - f_\Eps} \,.
\end{align}
\end{rmk}
A special choice of $\Lambda = \bar \Lambda_0 + \Eps \Lambda_1$ with $\bar \Lambda_0$ will give rise to a pure diffusion equation, as can be seen by letting $\Lambda$ to be a constant in Case III. The case of actual interest for further studies is of course when the tumbling rate takes the general form $\Lambda  (\frac z \Eps)$ as in section~\ref{sec:sasg}.

\section{Existence and Properties of $Q_0$}

In several scaling limits, the equilibrium $Q_0$ occurs as the solution of the  eigenfunction problem
\begin{align} \label{eq:Q-0-general}
      \del_y \vpran{(v \cdot G - y) Q_0}
= \Lambda_0(y) \vpran{\vint{Q_0} - Q_0},  \qquad  y\in R, \; v \in \VV  \,,
\\
Q_0 (y,v)  \geq 0, \qquad    \int_\R \int_\VV Q_0(y, v) \dy \dv = 1 \,. 
\end{align}
Here, we prove existence and uniqueness of $Q_0$ assuming that  $\Lambda_0$ may depend continuously on $y$ and satisfies, for  two constants $0 < \lambda_1 \leq \lambda_2$, the bounds
\begin{align} \label{bound:Lambda-appendix}
  0<   \lambda_1 \leq \Lambda_0(y) \leq \lambda_2 
\qquad
 \qquad \forall y \in \R \, .
\end{align}
We also derive some basic properties of $Q_0$. The precise statement  is 
\begin{thm}[Existence and regularity of $Q_0$] \label{thm:Q-0-general}
Suppose $\Lambda$ is continuous and satisfies \eqref{bound:Lambda-appendix}, then there exists a unique probability solution $Q_0$ of equation~\eqref{eq:Q-0-general}. Moreover, $Q_0$,  is compactly supported on $[-|G|, \, |G|] \times \VV$, strictly positive on $(-|G|, \, |G|) \times \VV$ and satisfies
\smallskip

\Ni (a) $\vint{Q_0} \in L^\infty(-|G|, \, |G|)$ and $Q_0$ is Lipschitz continuous in compact subsets of $\R \times  \VV \backslash \{ v \cdot G=y\}$;

\smallskip

\Ni (b) If $\lambda_2 < 1$, then $Q_0$ blows up at the diagonal $\{v \cdot G = y, \; |y| < |G|\}$,  with a rate no less than $\abs{v \cdot G - y}^{\lambda_2 - 1}$;

\smallskip

\Ni (c) If $\lambda_1 > 1$, then $Q_0 \in L^\infty([-|G|, \, |G|] \times \VV)$;

\smallskip 

\Ni (d) If $\lambda_1 < 1$, then $Q_0$ blows up at the diagonal at most as 
\begin{align*}
   Q_0(y, v)
\leq
   \frac{\lambda_2 (2 |G|)^{1-\lambda_1}}{1-\lambda_1} \norm{\vint{Q_0}}_{L^\infty} \abs{v \cdot G - y}^{\lambda_1 - 1} \,. 
\end{align*}
Furthermore, for $y$ fixed, $Q_0$ only depends on $v$ in the direction $ G$. In the special case where $\Lambda$ is an even function in $y$, $Q_0$ is even in $(y,v)$, 
 \begin{align} \label{symm-Q-0}
   Q_0(y, v) = Q_0(-y, -v) \,.
\end{align}
\end{thm}
\begin{rmk}
Additionally, one can check that for $\lambda_2=1$ the blow-up is at least logarithmic and for $\lambda_1 =1$, the blow-up is at most logarithmic.
\end{rmk}


\begin{proof}
The proof is organized as follows. We build a probability solution by a time evolution method and denote integration by $Q_0 \dv \dy$ rather than ${\rm d}Q_0(y,v)$. Then, we prove that such solution satisfies the announced properties. Finally we prove uniqueness.

\smallskip

\Ni \underline{{\bf Compact Support}} \, First we show that any probability solution to~\eqref{eq:Q-0-general} must be compactly supported on $[-|G|, \, |G|] \times \VV$. Indeed, integrating~\eqref{eq:Q-0-general} in $v$ gives
\begin{align*}
   \del_y \int_\VV (v \cdot G - y) Q_0 \dv = 0 \,.
\end{align*}
Therefore, there exists a constant $\alpha_0$ such that
\begin{align*}
   \int_\VV (v \cdot G - y) Q_0 \dv = \alpha_0 \,, \qquad \forall  y \in \R.
\end{align*}
We claim that $\alpha_0 = 0$ by the following observations:
\begin{align*}
   \alpha_0 &= \int_\VV (v \cdot G - y) Q_0 \dv \leq 0 
\qquad
  \text{for $y > |G|$} \,,
\\
   \alpha_0 &= \int_\VV (v \cdot G - y) Q_0 \dv \geq 0 
\qquad
  \text{for $y < -|G|$} \,.
\end{align*}
For any $y > |G|$, repeating the above argument, we have
\begin{align*}
   0 = \int_\VV (v \cdot G - y) Q_0 \dv \leq (|G| - y) \int_\VV Q_0(y, v) \dv \leq 0 \,.
\end{align*}
Therefore $\int_\VV Q_0(y, v) \dv = 0$ for any $y > |G|$ which implies $Q_0 (y, v) = 0$ if $y > |G|$. Similar argument applies when $y \leq -|G|$. Hence $Q_0$ is compactly supported on $[-|G|, \, |G|] \times \VV$.

\smallskip

\Ni \underline{{\bf Existence}} \, Second, we show the existence of a  solution of~\eqref{eq:Q-0-general} with a support contained in $[-|G|, \, |G|] \times \VV$. 
To this end, let $\Eps > 0$ be arbitrary and consider the evolution equation
\begin{align} \label{eq:evolution-Q}
   \Eps \del_t h_\Eps 
   + \del_y \vpran{(v \cdot G - y) h_\Eps}
= \Lambda_0(y) &\vpran{\vint{h_\Eps} - h_\Eps}  \,,
\\
h_\Eps \big|_{t=0} = h^{in}(y, v) \,. \label{eq:evolution-Q-initial}
\end{align}
where the initial data satisfies
\begin{align} \label{assump:evolution-initial}
   \Supp h^{in} \subseteq [-|G|, \, |G|] \times \VV \,,
\qquad
  \norm{h^{in}}_{L^1(\R \times \VV)} = 1\,,
\qquad
  h^{in} \geq 0 \,. 
\end{align}
The global existence of a non-negative solution to~\eqref{eq:evolution-Q} is a classical  matter. Moreover,
\begin{align*}
    \norm{h_\Eps(t, \cdot, \cdot)}_{L^1(\R \times \VV)} = 1, \qquad  \forall t \geq 0.
\end{align*}

We can show that $h_\Eps$ also has a support included in $[-|G|, \, |G|] \times \VV$. To this end, let $\phi \in C^1(\R)$ be a non-negative function  such that 
\begin{align*}
   \phi(y)
= \begin{cases}
    0 \,, & y \in [-|G|, |G|]  \,, \\[1pt]
    \text{increasing} \,, & y > |G| \,, \\[1pt]
    \text{decreasing} \,, & y < - |G|  \,.
    \end{cases}
\end{align*}
Because  we always have $\phi'(y) (v \cdot G - y) \leq 0$, multiplying $\phi(y)$ to~\eqref{eq:evolution-Q} and integrate in $(y, v)$ gives
\begin{align*}
   \Eps \frac{\rm d}{\dt} \int_\R \int_\VV \phi(y) h_\Eps(t, y, v) \dv\dy =  \int_\R \int_\VV \phi'(y) (v \cdot G - y) h_\Eps \dv\dy   \leq 0 \,.
\end{align*}
Using the compact support property for $h^{in}(y, v)$,  we conclude that 
\begin{align*}
  0 \leq \int_\R \int_\VV \phi(y) h_\Eps(t, y, v) \dv\dy  \leq 
   \int_\R \int_\VV \phi(y) h^{in}(y, v) \dv\dy =0,  \qquad   \forall t \geq 0 \; .
\end{align*}
The compact support property follows.

Now, consider the family of probability measures $\{h_\Eps\}$ on $[0, 1] \times \R \times \VV$.  Being compactly supported, it is tight. Therefore, there exists a probability measure $h_0(t, x, v)$ such that, after extraction of a subsequence, 
\begin{align*}
   h_\Eps \stackrel{\ast}{\rightharpoonup} h_0
\qquad
  \text{as $\Eps \to 0$.}
\end{align*}
Take the limit $\Eps \to 0$ in equation~\eqref{eq:evolution-Q}, we have
\begin{align*}
   \del_y \vpran{(v \cdot G - y) h_0}
= \Lambda_0(y) &\vpran{\vint{h_0} - h_0} \,.
\end{align*}
Define the probability measure $Q_0$ by 
\begin{align*}
   Q_0 = \int_0^1 h_0(t, y, v) \dt \,,
\end{align*} 
then it satisfies the equation~\eqref{eq:Q-0-general} with support contained in $[-|G|, \, |G|] \times \VV$.

The next steps are to prove that this measure is  an $L^1$ function with the announced properties. 

\smallskip

\Ni \underline{\bf  Some general bounds on $Q_0$} \, In the sequel,  we will make use of the following representation of any weak solution of~\eqref{eq:Q-0-general}. We first deal with the values of $v$ and $y$ such that $y<v\cdot G$. For $\lambda >0$ to be chosen later,  we multiply equation~\eqref{eq:Q-0-general}  by $(v\cdot G -y)^{-\lambda}$ and use the chain rule to obtain
\[
  \del_y \vpran{(v\cdot G - y)^{1-\lambda} Q_0}
=   (v\cdot G -y)^{-\lambda} \Lambda_0(y) \vint{Q_0} +(v\cdot G -y)^{-\lambda}  (\lambda - \Lambda_0(y)) Q_0 \, .
\]
The choices $\lambda = \lambda_2$ and $\lambda = \lambda_1$  yield successively 
\[
\del_y \vpran{(v\cdot G - y)^{1-\lambda_2} Q_0} \geq  (v\cdot G - y)^{-\lambda_2} \Lambda_0(y) \vint{Q_0} \,  ,
\]
\[
\del_y \vpran{(v\cdot G - y)^{1-\lambda_1} Q_0} \leq  (v\cdot G -y)^{-\lambda_1} \Lambda_0(y) \vint{Q_0}  \, .
\]
Integrating for  $z \in (-\infty, y)$ (in fact $ (-|G|, y)$ and we keep this bound $- |G|$ which is more convenient for later computations),  we find 
\begin{align} \label{prop2Q-0}
(v\cdot G - y)^{1-\lambda_2} Q_0 \geq \lambda_1 \int_{ - |G|}^y   (v\cdot G -z)^{-\lambda_2}\vint{Q_0} (z) \dz, \qquad y< v\cdot G \, ,
\end{align}
\begin{align} \label{prop3Q-0}
(v\cdot G - y)^{1-\lambda_1} Q_0 \leq \lambda_2 \int_{ - |G|}^y  (v\cdot G -z)^{-\lambda_1} \vint{Q_0} (z) \dz,  \qquad y< v\cdot G \, .
\end{align}

\smallskip

\Ni \underline{{\bf  $Q_0$ does not carry mass on the diagonal.}} \,
We conclude from \eqref{prop3Q-0} that $Q_0$ is  an $L^\infty_{loc}$ function away from the diagonal where $v \cdot G = y$. However, this is not enough because our arguments below require that $Q_0$ does not carry mass on the diagonal. To do so, decompose $Q_0=H_0+ \mu$ with $\mu$ a measure supported by $\{ y =v \cdot G \}$ and $H_0 \perp \mu$. Inserting this decomposition in the equation, we find
\begin{align*}
0= \vint{Q_0} (y)- \mu  \qquad \text{on the diagonal }\; \{ y =v \cdot G \}\, .
\end{align*}
Since $\vint{Q_0}$ is independent of $v$, its restriction on $y = v \cdot G$ as a measure is zero. Hence $\mu = 0$.


\smallskip

\Ni \underline{{\bf Bound $\vint{Q_0} \in L^\infty(-|G|, \, |G|)$}} \, We estimate $Q_0$ on the sets $v \cdot G - y > 0$ and $v \cdot G - y < 0$ separately since the diagonal does not carry mass. For $v \cdot G - y > 0$, we integrate inequality~\eqref{prop3Q-0} in $v$ and obtain
\[
  \int_{v \cdot G > y} Q_0(y, v) \dv
\leq  \lambda_2
\int_{v \cdot G > y} \int_{-|G|}^y \frac{ (v \cdot G - y)^{\lambda_1 - 1}}{ (v \cdot G - z)^{\lambda_1}} \vint{Q_0}(z) \dz \dv \,.
\]
The inner integral in $v$ is bounded as
\begin{align*}
  \int_{v \cdot G > y}  \frac{(v \cdot G - y)^{\lambda_1 - 1}}{(v \cdot G - z)^{\lambda_1}} \dv
\leq
  c_d \int_{y/|G|}^1 \frac{(w |G| - y)^{\lambda_1 - 1}}{(w |G| - z)^{\lambda_1}} \dw 
  =   \frac{c_d}{ |G|} \int_{y-z}^{|G|-z} \frac{(w+z-y)^{\lambda_1 - 1}}{w^{\lambda_1}} \dw \, ,
\end{align*}
and thus , for  $z < y$, we conclude, using $c_d$ as a constant which may change from line to line and depend on $d, \; G$, $\lambda_1$, $\lambda_2$, that
\begin{align} \label{est:integral}
 \int_{v \cdot G > y}  \frac{(v \cdot G - y)^{\lambda_1 - 1}}{(v \cdot G - z)^{\lambda_1}} \dv
\leq
 \frac{c_d}{ |G|} \int_{1}^{\frac{|G|-z}{y-z} } \frac{(w-1)^{\lambda_1 - 1}}{w^{\lambda_1}} \dw \leq   \frac{c_d}{ |G| \lambda_1}+  \frac{c_d}{ |G|} \ln \frac{|G| - z}{y - z}  \,.
\end{align}
Therefore, we find 
\begin{align} \label{bound:Duhamel-1}
  \int_{v \cdot G > y} Q_0(y, v) \dv
\leq c_d
 \int_{-|G|}^y  \ln \frac{|G| - z}{y - z} \vint{Q_0}(z) \dz  + c_d \,.
\end{align}

\smallskip

As a first step, we now show that  $\vint{Q_0} \in L^2(-|G|, \, |G|)$.  
To this end, the contribution to the $L^2$-norm of $Q_0$ on $\{v \cdot G - y > 0\}$ satisfies, using Jensen's inequality, 
\begin{align*}
  \int_{-|G|}^{|G|} \vpran{\int_{v \cdot G > y} Q_0(y, v) \dv}^2\dy &
\leq
 {c_d} \int_{-|G|}^{|G|}
  \vpran{\int_{-|G|}^y \vint{Q_0}(z) \ln \frac{|G| - z}{y - z} \dz}^2 \dy 
  + {c_d}
\\
& \leq
  c_d \int_{-|G|}^{|G|} \int_z^{|G|}
  \vint{Q_0}(z) \vpran{\ln \frac{|G| - z}{y - z}}^2 \dy \dz + c_d \nn
\\
& \leq
  c_d \int_{-|G|}^{|G|} \vint{Q_0}(z) \dz + c_d
\leq 
  c_d \,,
\end{align*}
where we have used the fact that $Q_0$ is a probability measure  and the last inequality holds because
\begin{align*}
   \int_z^{|G|} \vpran{\ln \frac{|G| - z}{y - z}}^2 \dy
= \int_0^{|G| - z} \vpran{\ln \frac{y}{|G| - z}}^2 \dy
= \vpran{|G| - z} \int_0^1 (\ln y)^2 \dy < \infty \,.
\end{align*}
The estimate for the integral  where $v \cdot G - y < 0$ is similar. Combining the two parts over $v \cdot G > y$ and $v \cdot G < y$, we conclude that $\vint{Q_0} \in L^2(-|G|, \, |G|)$.
\smallskip

 Building on the $L^2$-bound of $\vint{Q_0}$ we can now show that $\vint{Q_0} \in L^\infty(-|G|, \, |G|)$.
To this end, use~\eqref{bound:Duhamel-1} and the Cauchy-Schwarz inequality, 
\begin{align*}
  \int_{v \cdot G > y} Q_0(y, v) \dv
& \leq
  c_d \norm{\vint{Q_0}}_{L^2} 
  \vpran{\int_{-|G|}^y \vpran{\ln \frac{|G| - z}{y-z}}^2 \dz}^{1/2} + c_d
\leq 
  c_d \,,
\end{align*}
since
\begin{align*}
  \int_{-|G|}^y \vpran{\ln \frac{|G| - z}{y-z}}^2 \dz
\leq
  2 \int_{-|G|}^y (\ln (|G| - z))^2 \dz
  + 2 \int_{-|G|}^y (\ln (y-z))^2 \dz
\leq
  4 \int_0^{2|G|} (\ln z)^2 \dz < \infty \,.
\end{align*}
The $L^\infty$-bound of $\int_{v \cdot G < y} Q_0(y, v) \dv$ follows in a similar way and the details are omitted. Combining these two parts, we obtain that $\vint{Q_0} \in L^\infty(-|G|, \, |G|)$.

\smallskip

\Ni \underline{{\bf  $Q_0$ is Lipschitz continuous away from the diagonal.}}  We now prove continuity of $Q_0$ and thus we  also obtain that $Q_0=0$ for $|y| = |G|$ except for the diagonal points $v\cdot G= \pm  |G|$. Such behaviour of $Q_0$ is depicted in the numerical results in Section~\ref{sec:numerics}. The proof for continuity uses the representation formula for the solution. We re-write the $Q_0$-equation as
\begin{align*}
   \del_y \vpran{e^{\int_{-|G|}^y \frac{\Lambda_0(z) - 1}{v \cdot G - z} \dz} Q_0(y)} 
= e^{\int_{-|G|}^y \frac{\Lambda_0(z) - 1}{v \cdot G - z} \dz}
   \frac{\Lambda_0(y)}{v \cdot G - y} \vint{Q_0} (y)\,.
\end{align*}
Away from the diagonal $\{v\cdot G = y\}$, the exponential term is smooth in $v$ and continuous in $y$, and the right-hand side is in $L^\infty_{loc} (\!\dv\dy)$, which implies that $Q_0(y, v)$ is Lipschitz continuous on every compact set in $\R  \times \VV \setminus \{v \cdot G = y \}$.

\smallskip

\Ni \underline{\bf Strict positivity of $\vint{Q_0}$ and $Q_0$} \,   Again we first deal with the values of $v$ and $y$ such that $y<v\cdot G$. Using inequality~\eqref{prop2Q-0} one has 
\begin{align*}
\int_{v\cdot G>y} Q_0(y,v) \dv &\geq 
 \lambda_1 \int_{ z=- |G|}^y \int_{v\cdot G>y} \frac{(v\cdot G - y)^{\lambda_2-1}} { (v \cdot G -z)^{\lambda_2}} \dv \; \vint{Q_0} (z) \dz
  \, .
\end{align*}
As before we  begin with the inner integral and estimate it, for $z\leq y$, as 
\begin{align*}
  \int_{v\cdot G>y} \frac{(v\cdot G - y)^{\lambda_2-1}} { (v \cdot G -z)^{\lambda_2}} \dv
&\geq 
  c_d \int_{v\cdot G>y} (v\cdot G - y)^{\lambda_2-1} \dv
\\
&= c_d \int_{y}^{|G|} (w - y)^{\lambda_2-1}  \vpran{|G|^2 - |w|^2}^{\frac{d-2}{2}} \dw
\geq
  c_d \vpran{|G| - y}^{\lambda_2 + \frac{d-2}{2}} \,,
\end{align*}
where $c_d$ is again a constant depending on $d, \; G$, $\lambda_1$, $\lambda_2$ which changes from line to line.
This yields
\[
\int_{v\cdot G>y} Q_0(y,v) dv \geq c_d (|G|-y)^{\lambda_2 + \frac{d-2}{2}}  \int_{- |G|}^y  \vint{Q_0} (z) dz \, .
\]
The same calculation, for $y>v \cdot G $ gives, 
\[
\int_{v\cdot G < y} Q_0(y,v) dv \geq c_d (|G| + y)^{\lambda_2 + \frac{d-2}{2}}  \int_{ y}^{|G|}  \vint{Q_0} (z) dz \, .
\]
Finally we arrive at 
\begin{align}\label{positivityaver}
\vint{Q_0} \geq c_d (|G| - |y| )^{\lambda_2 + \frac{d-2}{2}}  \int_{ - |G|}^{|G|}  \vint{Q_0} (z) dz = c_d (|G| - |y| )^{\lambda_2 + \frac{d-2}{2}} \, .
\end{align}
The strict positivity of $Q_0$ then follows from the lower bound on $\vint{Q_0}$ and \eqref{prop2Q-0}.

\smallskip

\Ni \underline{\bf Upper bound for $Q_0$}  Using the upper bound on $\vint{Q_0}$ and \eqref{prop3Q-0}, we conclude that for $y< v\cdot G $
\begin{align*} 
 Q_0 \leq \lambda_2  \norm{\vint{Q_0}}_{L^\infty}    (v\cdot G - y)^{\lambda_1-1} \int_{ - |G|}^y   (v\cdot G -z)^{-\lambda_1} dz =
   \frac{\lambda_2}{1-\lambda_1} \norm{\vint{Q_0}}_{L^\infty} \vpran{\vpran{\frac{v \cdot G + |G|}{v \cdot G - y}}^{1-\lambda_1} - 1}
\end{align*}
and a similar inequality for $y > v\cdot G $ concludes the points (c) and (d).

\smallskip


\smallskip

\Ni \underline{{\bf Uniqueness, $v\cdot G$ dependency  and symmetry}} \, Finally, we show that  a probability solution to~\eqref{eq:Q-0-general} must be unique. Suppose instead there are two probability solutions, denoted as $Q_1$ and $Q_2$.  From the steps above, these are $L^1$ functions. To avoid the difficulty that  $Q_1$, $Q_2$ vanish near the boundary $y = \pm G$, we introduce
\begin{align*}
   Q_3 = \frac{1}{2} \vpran{Q_1 + Q_2} \,.
\end{align*}
The main advantage of $Q_3$ is the uniform boundedness given by
\begin{align*}
   \frac{Q_1}{Q_3} \leq 2 \,.
\end{align*}
Such bound is elusive a-priori for $Q_1/Q_2$, which renders integration by parts involving $Q_1/Q_2$ invalid. Note that by the linearity of~\eqref{eq:Q-0-general}, $Q_3$ is also a normalized non-negative solution which satisfies all the properties shown above. In particular, $Q_3$ is strictly positive in $\big((-|G|, \, |G|) \times \VV\big)$. The uniqueness is shown by a similar argument as proving the positivity of $D_{0, 2}$ in Theorem~\ref{thm:front-asymp}. In particular, to make use of the entropy-type estimate, we multiply $Q_1/Q_3$ to the $Q_1$-equation and integrate in $(y, v)$. Since $Q_1, Q_3$ are both solutions to~\eqref{eq:Q-0-general}, we apply the same estimate as in the proof of Lemma~\ref{lem:positivity-entropy} and obtain that
\begin{align*}
   \int_{-|G|}^{|G|} \int_\VV \Lambda(y) Q_3(y, v) 
   \vpran{\frac{Q_1}{Q_3} - \vint{\frac{Q_1}{Q_3}}}^2 \dv\dy = 0 \,.
\end{align*}
Therefore, $Q_1/Q_3$ is independent of $v$. Hence there exists a function $\gamma(y)$ such that
\begin{align*}
    Q_1 = \gamma(y) Q_3 \,.
\end{align*}
Inserting such relation in the $Q_1$-equation gives
\begin{align*}
   \del_y \vpran{(v \cdot G - y) \gamma(y) Q_3} 
= \gamma(y) \Lambda(y) \vpran{\vint{Q_3} - Q_3} \,.
\end{align*}
We can now use the $Q_3$-equation to derive that
\begin{align*}
    \gamma'(y) (v \cdot G - y) Q_3 = 0 \,.
\end{align*}
This implies that $\gamma$ is a constant function in $y$. If we denote it as $\gamma_0$, then $Q_1 = \gamma_0 Q_3$. By the normalization conditions for both $Q_1, Q_3$, we get $Q_1 = Q_3$, which further implies that $Q_1 = Q_2$, therefore the uniqueness.

\smallskip

From uniqueness, the symmetry in~\eqref{symm-Q-0} follows immediately since $Q_0(-y, -v)$ is also a non-negative and normalized solution to~\eqref{eq:Q-0-general}. To show that $Q_0$ depends only on $v\cdot G$, we remark that the  solution in $d=1$, with $v_1$ the direction of $G$ can be extended to a solution in $d$ dimension in $v$ (independent of the orthogonal directions to $v_1$) and this provides the unique $d$-dimensional solution. 
\end{proof}


\section{Numerical illustration on $Q_0$}\label{sec:numerics}


We now numerically illustrate the properties established previously on the solution $Q_0$ of equation~\eqref{eq:Q-0-general} with $d=1$ which means $\VV = (-1,1)$. 
We also make the connection with the coefficients found for the continuum FLKS limits and compute 
 the flux $c_{0,2}$ defined by (\ref{def:c-0-2}).

In order to calculate $c_{0,2}$, we first solve the leading order equation ~\eqref{eq:Q-0-general} to obtain $Q_0$, and then we solve the following equation of the next order by using $Q_0(y,v)$, i.e.,
\begin{equation}\label{eq_inh}
	\partial_y((vG-y)h)=\Lambda_0(y)(<h>-h)+\Lambda_1(y)(<Q_0>-Q_0),
\end{equation}
with the constraint
\begin{equation}\label{eq_subh}
	\int_{-\infty}^\infty\int_{-1}^1h(y,v)dvdy=0.
\end{equation}
The flux $c_{0,2}$ is finally obtained by the integration of the solution $h(y,v)$ such that
\begin{equation}\label{eq_c0_h}
c_{0,2}=\frac{1}{2}\int_{-\infty}^\infty\int_{-1}^1	h(y,v)dvdy.
\end{equation}
We note that in the case when $\Lambda_0(y)$ is a constant, i.e., $\Lambda_0(y)=\bar \Lambda_0$, the flux $c_{0,2}$ is directly calculated from $Q_0$ by
\begin{equation}\label{eq_c0_Q}
\begin{split}
c_{0,2}=-\frac{1}{2\bar \Lambda_0}\int_{-\infty}^\infty \int_{-1}^1v\Lambda_1(y)Q_0(y,v)dvdy.
\end{split}
\end{equation}
%

\subsection{Numerical Scheme}
We consider the lattice mesh system on the domain $[-G,G] \times [-1,1] \subset \mathrm{R}^2$ such that
\begin{equation}
y_i=-G+ i\Delta y,\quad (i=0,\cdots,2I),
\end{equation}
\begin{equation}
v_j=-1+j\Delta v,\quad (j=0,\cdots,2 J),
\end{equation}
with $\Delta y=G/I$ and $\Delta v=1/J$.

We choose the mesh system such that for each $y_i$, there exists a mesh point on the diagonal $y=vG$. More precisely, $J/I$ is an integer. Denote the mesh point on the diagonal by $(y_{i_*},v_{j_*})$. For each $v_j$, depending on if $y_i\leq v_jG$ or $y_i\geq v_jG$, the discretizations at the grid point $(y_i,v_j)$ are different. The details are as following.  

For fixed $j$,
when $y_i\le v_jG$, equation~\eqref{eq:Q-0-general} is discretized by using the first- and second-order backward difference scheme , i.e., 
\begin{subequations}\label{eq_backdiff}
\begin{align}
&\frac{W_{1,j}Q_{1,j}-W_{0,0}Q_{0,j}}{\Delta y}=\Lambda_0^1\left(<Q>_1-Q_{1,j}\right),\\ 
&\frac{3W_{i,j}Q_{i,j}-4W_{i-1,j}Q_{i-1,j}+W_{i-2,j}Q_{i-2,j}}{2\Delta y}=\Lambda_0^i \left(<Q>_i-Q_{i,j}\right)\quad \mathrm{for}\quad y_2\le y_i\le v_jG,
\end{align}
\end{subequations}
and for $y_i\ge v_jG$, using the first- and second-order forward difference scheme, i.e.,
\begin{subequations}\label{eq_forwdiff}
\begin{align}
&\frac{W_{I,j}Q_{I,j}-W_{I-1,j}Q_{I-1,j}}{\Delta y}
=\Lambda_0^{I-1}\left(<Q>_{I-1}-Q_{I-1,j}\right),\\
&\frac{-3W_{i,j}Q_{i,j}+4W_{i+1,j}Q_{i+1,j}-W_{i+2,j}Q_{i+2,j}}{2\Delta y}=\Lambda_0^i \left(<Q>_i-Q_{i,j}\right)\quad \mathrm{for}\quad v_jG\le y_i\le y_{I-2},
\end{align}
\end{subequations}
with the boundary condition
\begin{equation}\label{eq_bound}
Q_{0,j}=Q_{2I,j}=0.	
\end{equation}
Here, we write $Q_{i,j}=Q_0(y_i,v_j)$, $<Q>_i=<Q_0>(y_i)$, $\Lambda_0^i=\Lambda_0(y_i)$, and
\begin{equation}\label{eq_W}
	W_{i,j}=v_j G-y_i.
\end{equation}
%

At each mesh point on the diagonal ($y_{i_*}$,$v_{j_*}$), we calculate two values for $Q_{i_*,j_*}$ by using equation~(\ref{eq_backdiff}) and (\ref{eq_forwdiff}) with letting $W_{i_*,j_*}Q_{i_*,j_*}=0$ on each eqaution.
The values calculated at the left side of the diagonal by (\ref{eq_backdiff}), say 
$Q^L_{i_*}$, are used in the integration $\int_{v_{j_*}}^{1} Q_0(y_i,v')dv'$ and the values calculated at the right side of the diagonal by (\ref{eq_forwdiff}), say $Q^R_{i_*}$, are used in the integration $\int_{-1}^{v_{j_*}}Q_0(y_i,v')dv'$.

Thus, we calculate $<Q>_i$ ($i=1,\cdots,2I-1$) as
\begin{equation}\label{simpson_qy}
<Q>_i=\frac{\Delta v}{4}
\left(
 Q_{i,0}+2\sum_{j=1}^{j_*-1}Q_{i,j}+Q^R_i
+Q^L_i+2\sum_{j=j_*+1}^{2J-1}Q_{i,j}+Q_{i,2J}
 \right),
\end{equation}
and $<Q>_0=<Q>_{2I}=0$.
Hereafter, the trapezoidal rule is used for the integration with respect to $v$ while Simpson's rule is used for the integration with respect to $y$.
Thus, for example, the integration $\overline{Q}=\int_{-G}^G <Q_0>(y)dy$ is calculated as
\begin{equation}\label{simpson_q}
\overline{Q}=\frac{\Delta y}{3}\left(<Q>_0+4\sum_{i=1}^{I}<Q>_{2i-1}+2\sum_{i=1}^{I-1}<Q>_{2i}+<Q>_{2I}\right).
\end{equation}

To obtain the solution of Eqs.~(\ref{eq_backdiff})--(\ref{eq_bound}), we consider the following time-evolution semi-implicit scheme,
\begin{subequations}\label{eq_evol_backward}
\begin{align}
&\frac{Q^{n+1}_{1,j}-Q^n_{1,j}}{\Delta t}=
-\frac{W_{1,j}Q^{n+1}_{1,j}-W_{0,j}Q^{n+1}_{0,j}}{\Delta y}+\Lambda_0^1\left(<Q>^n_1-Q_{1,j}^{n+1}\right),\\
&\frac{Q^{n+1}_{i,j}-Q^n_{i,j}}{\Delta t}=
-\frac{3W_{i,j}Q^{n+1}_{i,j}-4W_{i-1,j}Q^{n+1}_{i-1,j}+W_{i-2,j}Q^{n+1}_{i-2,j}}{2\Delta y}+\Lambda_0^i\left(<Q>^n_i-Q^{n+1}_{i,j}\right),
\end{align}
\end{subequations}
for $y_1\le y_i\le v_j G$ with a uniform initial condition.
This scheme uses a lower diagonal matrix which solves $Q^{n+1}_{i,j}$ very quickly.
For $v_j G\le y_i\le y_{2I-1}$, we replace the backward difference in (\ref{eq_evol_backward}) with the forward difference (\ref{eq_forwdiff}), i.e.,
\begin{subequations}\label{eq_evol_forward}
\begin{align}
&\frac{Q^{n+1}_{2I-1,j}-Q^n_{2I-1,j}}{\Delta t}=
-\frac{W_{I,j}Q^{n+1}_{2I,j}-W_{2I-1,j}Q^{n+1}_{2I-1,j}}{\Delta y}+\Lambda_0^{2I-1}\left(<Q>^n_{2I-1}-Q^{n+1}_{2I-1,j}\right),\\
&\frac{Q^{n+1}_{i,j}-Q^n_{i,j}}{\Delta t}=
-\frac{-3W_{i,j}Q^{n+1}_{i,j}+4W_{i+1,j}Q^{n+1}_{i+1,j}-W_{i+2,j}Q^{n+1}_{i+2,j}}{2\Delta y}+\Lambda_0^i\left(<Q>^n_i-Q^{n+1}_{i,j}\right),
\end{align}
\end{subequations}
which solves $Q^{n+1}_{i,j}$ by using the upper diagonal matrix.

At each time step, we calculate $\overline{Q}^n$ by equation (\ref{simpson_q}) and normalize the solution of the time-evolution scheme (\ref{eq_evol_backward}) and (\ref{eq_evol_forward}), say $\tilde Q^n_{i,j}$, by
\begin{equation}
	Q^n_{i,j}=\tilde Q^n_{i,j}/\overline{Q}^n,
\end{equation}
in order to satisfy the normalized condition $\int_{-G}^G\int_{-1}^1Q(y,v)dydv=1$.

We repeat the above process until the $Q_{i,j}^n$ satisfy the following convergence condition 
\begin{equation}\label{eq_conv}
\sum_{i,j}|Q_{i,j}^n-Q_{i,j}^{n-100}|\Delta v\Delta y < 10^{-10}.
\end{equation}

After we obtain the numerical solution $Q_{i,j}$, we solve equation (\ref{eq_inh}) by using the same time-evolution scheme as (\ref{eq_evol_backward}) and (\ref{eq_evol_forward}) with the inhomogeneous term $\Lambda_1^i(<Q>_i-Q_{i,j})$.
In order to satisfy the condition (\ref{eq_subh}), we correct the solution of the time-evolution scheme at each time step, say $\tilde h^n_{i,j}$, by 
\begin{equation}\label{eq_correct_h}
	h^n_{i,j}=\tilde h^n_{i,j}-rQ_{i,j},
\end{equation}
where $r$ is calculated as $r=\int_{-G}^G\int_{-1}^1\tilde h^n(y,v)dydv$.

\subsection{Numerical Result}

We carry out the numerical computation of  (\ref{eq_evol_backward})--(\ref{eq_correct_h}) when the tumbling functions are written as,
\begin{equation}\label{eq_lambda0}
\Lambda_0(y)=\bar \Lambda_0(1-\chi \tanh(y)),
\end{equation}
 and 
 \begin{equation}\label{eq_lambda1}
 \Lambda_1(y)=-\tanh(y).
 \end{equation}
Here, $\bar \Lambda_0$ is the mean tumbling rate and $\chi$ is the modulation amplitude of $\Lambda_0(y)$.

\begin{table}[htbp]
\centering
	\begin{tabular}{c cc c cc}
	\hline\hline
	        & \multicolumn{2}{c}{$\bar \Lambda_0$=2.0} & &\multicolumn{2}{c}{$\bar \Lambda_0$=0.5}\\
	        \cline{2-3} \cline{5-6}
		$I$ & $\frac{|c^I_{0,2}-c^*_{0,2}|}{c_{0,2}^*}$ & $\displaystyle \max_{i}|<Q>^I_i-<Q>^*_i|$ &
		 & $\frac{|c^I_{0,2}-c^*_{0,2}|}{c_{0,2}^*}$ & $\displaystyle \max_{i}|<Q>^I_i-<Q>^*_i|$ \\
		\hline
		200 & 9.1$\times 10^{-5}$ & 1.2$\times 10^{-4}$ && 8.1$\times 10^{-4}$ & 4.2$\times 10^{-2}$\\
		400 & 2.2$\times 10^{-5}$ & 2.9$\times 10^{-5}$ && 5.8$\times 10^{-4}$ & 2.5$\times 10^{-2}$\\
		800 & 5.4$\times 10^{-6}$ & 6.9$\times 10^{-6}$ && 3.4$\times 10^{-4}$ & 1.3$\times 10^{-2}$\\
		1600 & 1.1$\times 10^{-6}$ & 1.4$\times 10^{-6}$ && 1.5$\times 10^{-4}$ & 5.5$\times 10^{-3}$\\
		\hline\hline
	\end{tabular}
\caption{
The numerical accuracy for $c_{0,2}$ calculated by equation~(\ref{eq_c0_Q}) and $<Q>_i$ for different values of $\bar \Lambda_0=2.0$ and 0.5, where the tumbling rate $\Lambda_0$ is considered, i.e., $\chi=0$ in equation~(\ref{eq_lambda0}).
The number of mesh intervals $\Delta v$ and the time step size $\Delta t$ are set as $J=I$ and $\Delta t=\Delta y/2G$, respectively.
The parameter of the external chemical gradient $G=1.0$ is fixed.
The reference quantities $c_{0,2}^*$ and $<Q>_i^*$ are the results obtained with $I=3200$.
}\label{t_accuracy}
\end{table}
In the numerical computation, we set the time-step size as $\Delta t=\Delta y/2G$ and the numbers of mesh interval $\Delta y$ and $\Delta v$ as $I=J$.
Table \ref{t_accuracy} shows the accuracy of the numerical results for various mesh systems in the case when $\Lambda_0$ is constant, i.e., $\Lambda_0=\bar\Lambda_0$ and $\chi=0$.
It is seen that the flux $c_{0,2}$ and the moment $<Q_0>$ converge against the mesh interval for both $\bar \Lambda_0$=2.0 and 0.5.
The accuracy of $c_{0,2}$ is approximately second order for $\bar \Lambda_0$=2.0 while it is approximately first order for $\bar \Lambda_0$=0.5.
We also compare the results of $c_{0,2}$ obtained by (\ref{eq_c0_h}) and by (\ref{eq_c0_Q}).
The relative differences of the two values are 1.5$\times 10^{-6}$ for $\bar \Lambda_0=2$ and 2.2$\times 10^{-2}$ for $\bar \Lambda_0=0.5$ in the fine mesh system with $I=3200$.
In the following of this section, we show the numerical results obtained in the mesh system with $I=1600$.

\begin{figure}[tb]
\centering
\includegraphics{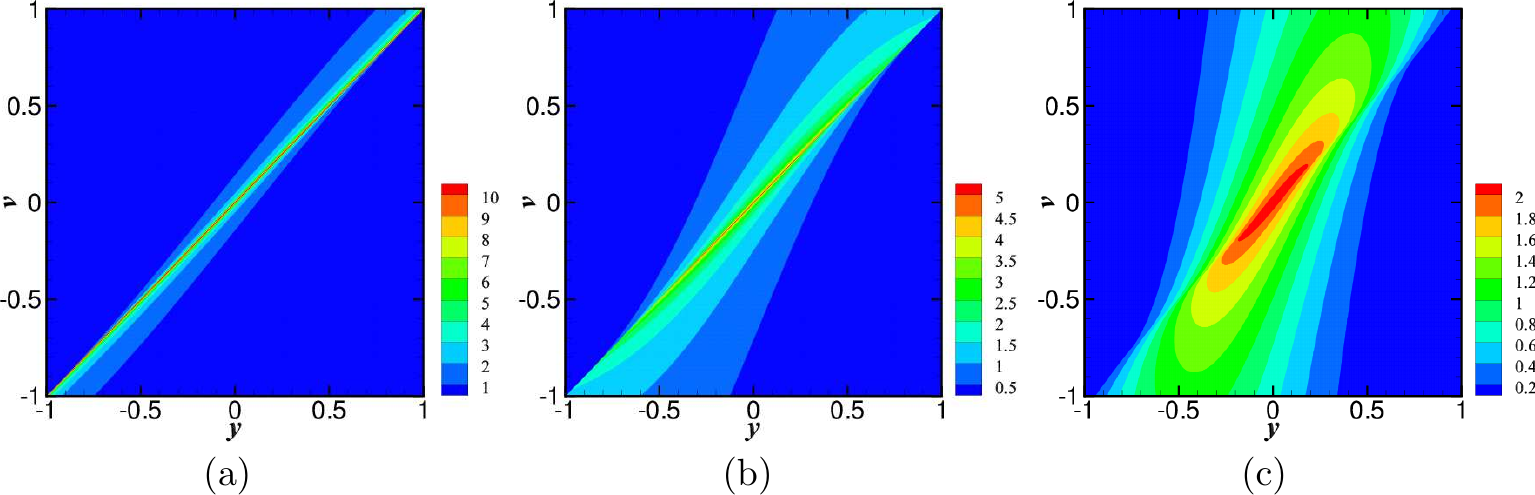}
\caption{
Distribution of $Q_0(y,v)$ for various values of $\bar \Lambda_0$, i.e., $\bar \Lambda_0=0.5$ in (a), $\bar \Lambda_0=1.0$ in (b), and $\bar \Lambda_0=2.0$ in (c).
The modulation amplitude $\chi$=0 and chemical gradient $G=1$ are fixed.
}\label{fig_2dmap}
\end{figure}
Figures \ref{fig_2dmap} shows the numerical results of $Q_0$   for different values of $\bar \Lambda_0$ when the tumbling rate $\Lambda_0(y)$ is constant, i.e, $\chi=0$.
It is seen that the distributions are origin symmetry, i.e., $Q(y,v)=Q(-y,-v)$ because the tumbling rate $\Lambda_0$ is independent of the internal state $y$.
As the tumbling rate $\bar \Lambda_0$ decreases, $Q_0$ becomes more and more concentrated at the diagonal $y=vG$.

\begin{figure}[tb]
\centering
\includegraphics[width=15cm]{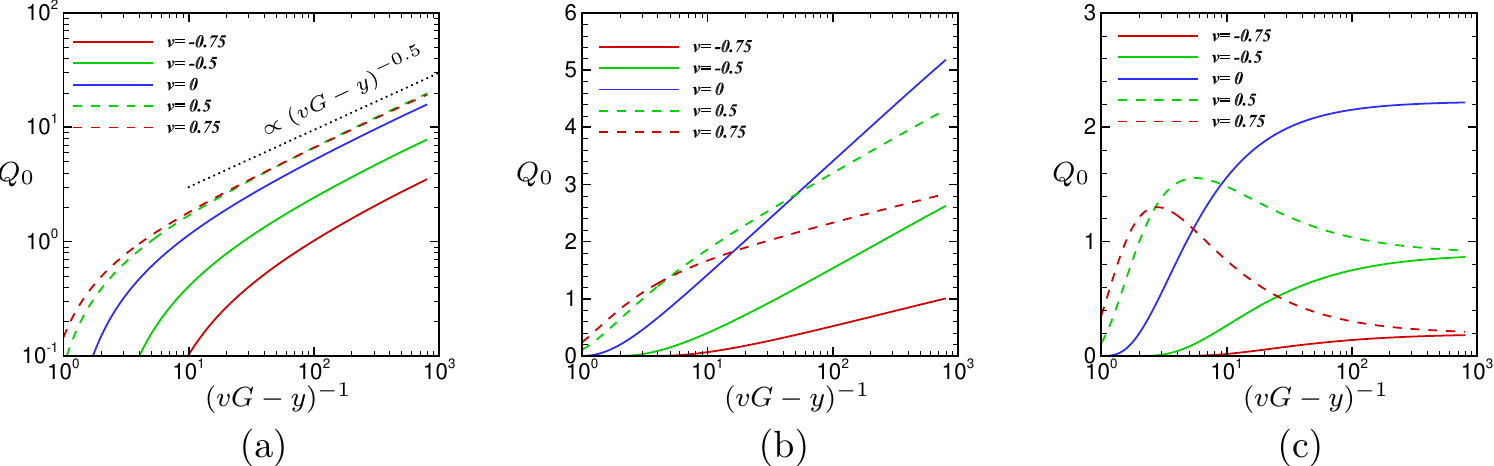}
\caption{
The blow-up behaviors near the diagonal for $vG-y>0$.
The distributions of $Q_0(y,v)$ against $(vG-y)^{-1}$ when the velocity $v$ is fixed are plotted for $\bar \Lambda_0$=0.5 in (a), $\bar \Lambda_0$=1.0 in (b), and $\bar \Lambda_0$=2.0 in (c).
The modulation amplitude $chi$=0 and chemical gradient $G$=1 are fixed.
Figure (a) is shown in the double logarithmic while Figures (b) and (c) are shown in the single logarithmic.
Note that the distributions for $v$=0.5 and 0.75 (dashed lines) correspond to those for $v=-0.5$ and $-0.75$ in the domain $vG-y<0$, respectively, due to the symmetry $Q_0(y,v)=Q_0(-y,-v)$.
}\label{fig_blowup}
\end{figure}

Figure \ref{fig_blowup} shows the one-dimensional distribution of $Q_0(y,v)$ against $(vG-y)^{-1}$ in the domain $vG-y>0$ when the velocity $v$ is fixed.
It is clearly seen that $Q_0(y,v)$ blows up at the diagonal for $\bar \Lambda_0\le 1$; the rates of the blowup are $(vG-y)^{\bar \Lambda_0-1}$ for $\bar \Lambda_0=0.5$ and logarithmic for $\bar \Lambda_0=1.0$ as is established in Theorem~\ref{thm:Q-0-general}.
On the other hand, for $\bar Lambda_0$=2.0, $Q_0$ is continuous at the diagonal $y=vG$.
This is seen because the curves for $v$=0.75 and 0.5 (dashed lines) correspond to those for $v=-0.75$ and $-0.5$ in $vG-y<0$, respectively, due to the symmetry and the curves with the same values of $|v|$ converge to the same values at the diagonal $y=vG$.

Figure \ref{fig_decay} shows the decay behavior of $Q_0(y,v)$ near $y=-G$ for a fixed velocity $v$.
These plots illustrate the behavior $Q_0(y,v)=(|G|-|y|)^\alpha$ near $y=-G$ with $\alpha>1$.
The power $\alpha$ is measured as $\alpha=1.6$ for $\bar \Lambda_0=0.5$, $\alpha=1.8$ for $\bar \Lambda_0=1.0$, and $\alpha=2.6$ for $\bar \Lambda_0=2.0$.
Thus, the power increases with $\bar \Lambda_0$.
All these numerical results illustrate the properties of $Q_0(y,v)$ established in Theorem~\ref{thm:Q-0-general}.

\begin{figure}[tb]
\centering
\includegraphics[width=15cm]{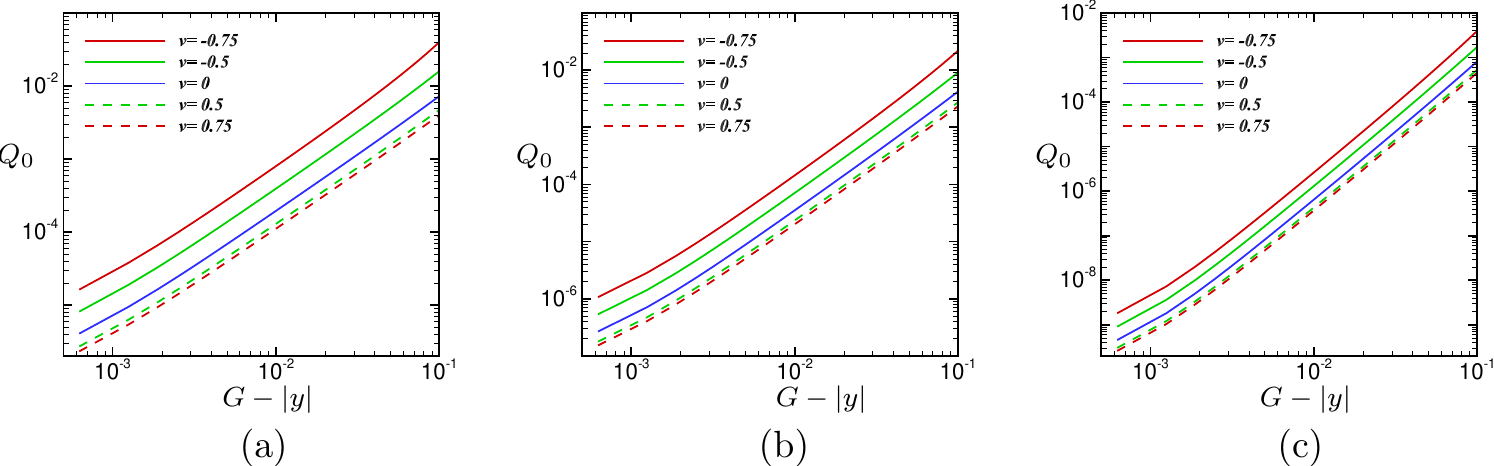}
\caption{
The decay behavior of $Q_0$ near $y=-G$.
The distributions of $Q_0(y,v)$ against $G-|y|$ when the velocity $v$ is fixed are plotted for $\bar \Lambda_0$=0.5 in (a), $\bar \Lambda_0$=1.0 in (b), and $\bar \Lambda_0$=2.0 in (c).
The modulation amplitude $chi$=0 and chemical gradient $G$=1 are fixed.
These plots illustrate the behavior $Q_0(y,v)=(|G|-|y|)^\alpha$ with $\alpha>1$ and $\alpha$ increases with $\bar \Lambda_0$.
}\label{fig_decay}
\end{figure}

\begin{figure}[tb]
\centering
\includegraphics[width=12cm]{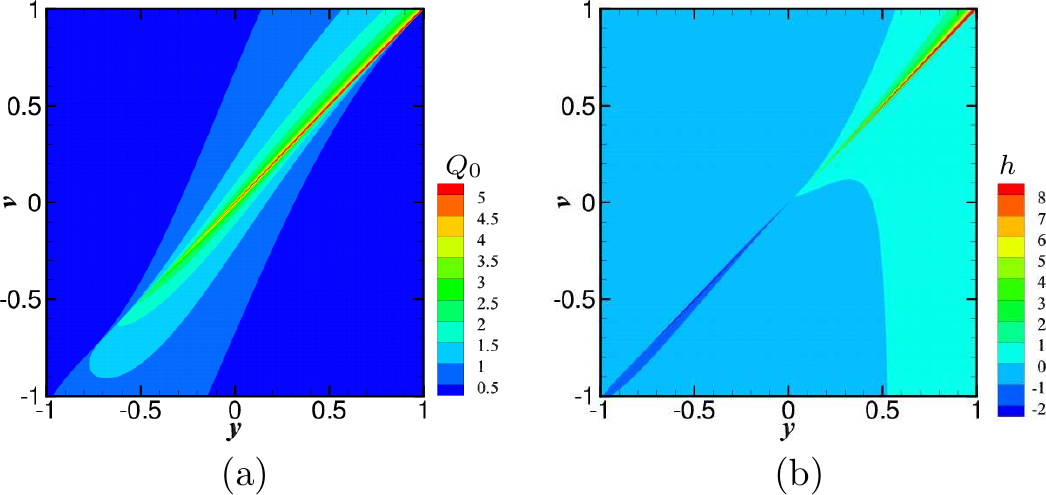}
\caption{
Distributions of $Q_0(y,v)$ in (a) and $h(y,v)$ in (b) when $\Lambda_0$ depends on $y$.
The mean tumbling rate $\bar \Lambda_0$=1 and chemical gradient $G$=1 are fixed.
}\label{fig_h_2dmap}
\end{figure}

Figure \ref{fig_h_2dmap} shows the distributions of $Q_0(y,v)$ and $h(y,v)$ when the tumbling rate $\Lambda_0$ depends on the internal state $y$.
It is clearly seen that the distributions are not anymore symmetric but rather concentrated along the diagonal $y=vG$ in $y>0$.

Finally, we show the drift velocity $c_{0,2}$ against the chemical gradient $G$ in Figure~\ref{fig_c0}.
Figure~\ref{fig_c0}(a) shows the results for various values of $\bar \Lambda_0$ when the tumbling rate is constant, i.e., $\chi$=0, while figure~\ref{fig_c0}(b) shows the results for various values of modulation $\chi$, where the tumbling rate $\Lambda_0$ is not constant.
In both cases, when the gradient is small, say $G\lesssim 1$, the fluxes $c_{0,2}$ are almost linearly proportional to the gradient, $c_0\propto G$.
However, they saturate for $G\gtrsim 10$ and approach to the constant values.
These results illustrate the boundedness of the drift velocity $c_{0,2}$ when the tumbling rate $\Lambda_0$ is constant.
Unexpectedly, we can also observe non-monotonic profiles of $c_{0,2}$ against the chemical gradient $G$ when the modulation is large, i.e., $\chi$=0.5 and 0.8.

\begin{figure}[tb]
\centering
\includegraphics[width=12cm]{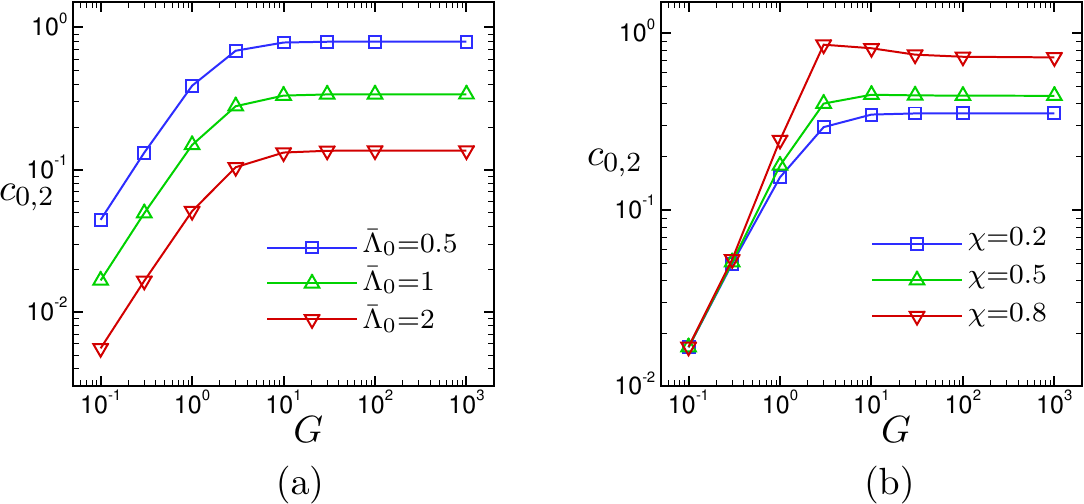}
\caption{
The drift velocity $c_{0,2}$ vs. the chemical gradient $G$.
The left figure (a) shows the results for various values of $\bar\Lambda_0$ when $\chi$=0 is fixed, where the tumbling rate $\Lambda_0(y)$ is constant.
The right figure (b) shows the results for various values of $\chi$ when $\bar\Lambda_0$=1 is fixed, where the tumbling rate $\Lambda_0(y)$ depends on the internal state $y$.
}\label{fig_c0}
\end{figure}

\section{Conclusion}

We have systematically studied scaling limits of kinetic equations which describe run-and-tumble movement of bacteria with internal chemical pathways. The complexity of the phenomena stems form the  different scales between space, time and velocity as well as the stiffness of the tumbling response and methylation adaptation. The question in these different asymptotic limits is to distinguish between the standard Keller-Segel equation and the flux-limitied Keller-Segel.

It appears that stiff-response,  a physically appropriate regime, always leads to the FLKS equation. In particular this  conclusion applies to fast adaptation and stiff response which corresponds to the measured paprameters for {\it E. coli}. 

The dominating profile, given by the function $Q_0(y,v)$, is also remarkable with a possible blow-up along the diagonal $y =v |G|$. However, the physically relevant regime is when $\Lambda_0 >1$ and then $Q_0(y,v)$ is  smooth. 

Several other possible scalings are still possible. Also we have ignored noise in the internal state, which can be a route to mathematically interesting analysis.

\appendix

\section{Well-posedness of the corrector equation~\eqref{eq:h}}
\label{appendix:well-posed}

We show that the equation~\eqref{eq:h}  for the corrector has a solution using the Fredholm theory. More generally, we consider the equation with a source term $R_1$ such that
\begin{align} \label{eq:h-general}
   \del_y \vpran{(v \cdot G - y) h} 
= \Lambda(y) (\vint{h} - h) + R_1(y, v) \,,
\qquad
   h(\pm |G|, v) = 0 \,. 
\end{align}

The well-posedness theorem states
\begin{thm} \label{thm:Fredholm}
Let $Q_0$ be the unique non-negative solution given in Theorem~\ref{thm:Q-0-general}. Suppose the function $R_1$ satisfied that $R_1 \in L^1((-|G|, \, |G|) \times \VV)$ and 
\begin{align} \label{cond:R-1}
   \int_\VV\int_{-|G|}^{|G|} R_1(y, v) \dy\dv = 0 \,,
\qquad
\end{align} 
and there exist constants $\mu_0 > -1$ and $c_0 > 0$ such that 
\begin{align} \label{bound:R-1}
   |R_1| \leq c_0 \vpran{\abs{v \cdot G - y}^{\mu_0} + 1} \,, 
\qquad
   (y, v) \in (-|G|, \, |G|) \times \VV \,.
\end{align}
Then equation~\eqref{eq:h-general} has a unique solution $h \in L^1((-|G|, \, |G|) \times \VV)$ such that 
\begin{align} \label{cond:h-orthog}
   \int_\VV\int_{-|G|}^{|G|} h(y, v) \dy\dv = 0 \,.
\end{align}

\end{thm}
\begin{proof}

Re-write~\eqref{eq:h-general} as
\begin{align*}
   (v \cdot G - y) \del_y h + (\Lambda(y) - 1) h 
= \Lambda(y) \vint{h} + R_1 \,.
\end{align*}
Then for $v \cdot G - y > 0$, we have
\begin{align*}
   \del_y \vpran{e^{\int_{-|G|}^y \frac{\Lambda(z) - 1}{v \cdot G - z} \dz} h} 
= e^{\int_{-|G|}^y \frac{\Lambda(z) - 1}{v \cdot G - z} \dz}
   \frac{\Lambda(y)}{v \cdot G - y} \vint{h}   + e^{\int_{-|G|}^y \frac{\Lambda(z) - 1}{v \cdot G - z} \dz}
      \frac{1}{v \cdot G - y} R_1 \,,
\end{align*}
which gives 
\begin{align} \label{eq:h-integ-positive}
   h(y, v)
&= \int_{-|G|}^y 
   e^{-\int_z^y \frac{\Lambda(w) - 1}{v \cdot G - w} \dw}
   \frac{1}{v \cdot G - z}   \vpran{\Lambda(z)
\vint{h}(z) + R_1(z, v)} \dz 
\qquad
   v \cdot G - y > 0 \,.
\end{align}
Similarly, if $v \cdot G - y < 0$, then
\begin{align*}
   \del_y \vpran{e^{\int_y^{|G|} \frac{\Lambda(z) - 1}{z - v \cdot G} \dz} h} 
= - e^{\int_y^{|G|} \frac{\Lambda(z) - 1}{z - v \cdot G} \dz}
   \frac{\Lambda(y)}{y - v \cdot G} \vint{h}
   - e^{\int_y^{|G|} \frac{\Lambda(z) - 1}{z - v \cdot G} \dz}
      \frac{1}{y - v \cdot G} R_1 \,,
\end{align*}
which gives 
\begin{align} \label{eq:h-integ-negative}
   h(y, v)
&= \int_y^{|G|}
   e^{-\int_y^z \frac{\Lambda(w) - 1}{w - v \cdot G} \dw}
   \frac{1}{z - v \cdot G} 
   \vpran{\Lambda(z)\vint{h}(z) + R_1(z, v)}\dz 
\qquad
   v \cdot G - y < 0 \,.
\end{align}
Using~\eqref{eq:h-integ-positive} and \eqref{eq:h-integ-negative}, we obtain the integral equation for $\vint{h}$ as
\begin{align*}
   \vint{h}(y)
&= \frac{1}{2} \int_{\VV} \int_{-G}^y \One_{v \cdot G > y}
   e^{-\int_z^y \frac{\Lambda(w) - 1}{v \cdot G - w} \dw}
   \frac{1}{v \cdot G - z} 
   \vpran{ \Lambda(z) \vint{h}(z) + R_1(z, v)} \dz
\\
& \quad \, 
   + \frac{1}{2} \int_{\VV} \int_y^{|G|} \One_{v \cdot G < y}
   e^{-\int_y^z \frac{\Lambda(w) - 1}{w - v \cdot G} \dw}
   \frac{1}{z - v \cdot G} 
   \vpran{\Lambda(z) \vint{h}(z) + R_1(z, v)}\dz \,.
\end{align*}
Denote $K$ as the linear operator such that
\begin{align*}
   K \vint{h}
&= \frac{1}{2} \int_{\VV} \int_{-G}^y \One_{v \cdot G > y}
   e^{-\int_z^y \frac{\Lambda(w) - 1}{v \cdot G - w} \dw}
   \frac{\Lambda(z)}{v \cdot G - z} 
   \vint{h}(z) \dz\dv
\\
& \quad \, 
   + \frac{1}{2} \int_{\VV} \int_y^{|G|} \One_{v \cdot G < y}
   e^{-\int_y^z \frac{\Lambda(w) - 1}{w - v \cdot G} \dw}
   \frac{\Lambda(z)}{z - v \cdot G} 
   \vint{h}(z) \dz\dv \,.
\end{align*}
Denote the source term as
\begin{align} \label{def:R}
   R(y)
&= \frac{1}{2} \int_{\VV} \int_{-|G|}^y \One_{v \cdot G > y}
   e^{-\int_z^y \frac{\Lambda(w) - 1}{v \cdot G - w} \dw}
   \frac{1}{v \cdot G - z} 
   R_1(z, v) \dz\dv   \nn
\\
& \quad \, 
   + \frac{1}{2} \int_{\VV} \int_y^{|G|} \One_{v \cdot G < y}
   e^{-\int_y^z \frac{\Lambda(w) - 1}{w - v \cdot G} \dw}
   \frac{1}{z - v \cdot G} 
   R_1(z, v) \dz\dv \,.
\end{align}
Using the bound of $R_1$ in~\eqref{bound:R-1} and the estimate, for $v \cdot G - y > 0$,
\begin{align} \label{bound:expo-1}
   e^{\int_y^z \frac{\Lambda(w) - 1}{v \cdot G - w} \dw}
   \frac{1}{v \cdot G - z}
\leq
   \frac{1}{v \cdot G - y}
   e^{-\lambda_1 \int_z^y \frac{1}{v \cdot G - w} \dw}
= \frac{(v \cdot G - y)^{\lambda_1 - 1}}{(v \cdot G - z)^{\lambda_1}} \, , 
\end{align} 
we bound the first term in $R$ as
\begin{align*}
& \quad \,
   \abs{\frac{1}{2} \int_{\VV} \int_{-|G|}^y \One_{v \cdot G > y}
   e^{-\int_z^y \frac{\Lambda(w) - 1}{v \cdot G - w} \dw}
   \frac{1}{v \cdot G - z} 
   R_1(z, v) \dz\dv }
\\
&\leq
   c_d \int_\VV \int_{-|G|}^y \One_{v \cdot G > y}
   \frac{(v \cdot G - y)^{\lambda_1-1}}{(v \cdot G - z)^{\lambda_1}}
   \vpran{\vpran{v \cdot G - z}^{\mu_0} + 1} \dz \dv
\\
& \leq
    c_d \int_\VV \int_{-|G|}^y \One_{v \cdot G > y}
   \frac{(v \cdot G - y)^{\lambda_1-1}}{(v \cdot G - z)^{\lambda_1 - \mu_0}} \dz \dv + c_d
\qquad  \text{by~\eqref{est:integral}}
\end{align*}
Similar as the estimates in~\eqref{est:integral}, we have
\begin{align*}
& \quad \,
   \int_\VV \One_{z<y} \One_{v \cdot G > y}
   \frac{(v \cdot G - y)^{\lambda_1-1}}{(v \cdot G - z)^{\lambda_1 - \mu_0}} \dv
=  c_d \One_{z<y}  \int_{y/|G|}^1 
   \frac{(w |G| - y)^{\lambda_1-1}}{(w |G| - z)^{\lambda_1 - \mu_0}} \dw
\\
& \leq \frac{c_d}{\lambda_1 |G|} \One_{z<y} 
       \vpran{\frac{(w |G| - y)^{\lambda_1}}{(w |G| - z)^{\lambda_1 - \mu_0}} \Big|^1_{y/|G|}
       + \abs{\lambda_1 - \mu_0} |G|
          \int_{y/|G|}^1 \frac{(w |G| - y)^{\lambda_1}}{(w |G| - z)^{\lambda_1 - \mu_0 + 1}} \dw}
\\
& \leq
   c_d \vpran{(|G| - z)^{\mu_0} 
       + \abs{\lambda_1 - \mu_0} |G|
          \int_{y/|G|}^1 \frac{1}{(w |G| - z)^{- \mu_0 + 1}} \dw}
\\
& \leq  
   c_d (|G| - z)^{\mu_0}  \,,
\end{align*}
where again the constant $c_d$ may change from line to line. Since $\mu_0 > -1$, the first term in $R$ is bounded as
\begin{align*}
   \abs{\frac{1}{2} \int_{\VV} \int_{-|G|}^y \One_{v \cdot G > y}
   e^{-\int_z^y \frac{\Lambda(w) - 1}{v \cdot G - w} \dw}
   \frac{1}{v \cdot G - z} 
   R_1(z, v) \dz\dv }
\leq
   c_d \int_{-|G|}^{|G|} (|G| - z)^{\mu_0} \dz
< \infty \,.
\end{align*}
Similar estimates hold for the second term in $R$. Overall we have 
\begin{align*}
    R \in L^\infty(-|G|, \, |G|) \subseteq L^2(-|G|, \, |G|) \,.
\end{align*}
Moreover, $\vint{h}$ satisfies
\begin{align} \label{eq:vint-h}
   \vint{h} = K \vint{h} + R \,.
\end{align}
We will show that the operator $\Id - K$ is Fredholm and $R \in (\Null(\Id - K^\ast))^\perp$ where the orthogonality is taken in $L^2(-|G|, |G|)$. To this end, we first find the kernel of $K$ as
\begin{align*}
   k(y, z)
 = \frac{1}{2} \One_{z < y} \int_\VV \One_{v \cdot G > y}
   e^{-\int_z^y \frac{\Lambda(w) - 1}{v \cdot G - w} \dw}
   \frac{\Lambda(z)}{v \cdot G - z} \dv
   + \frac{1}{2} \One_{z > y} \int_\VV \One_{v \cdot G < y}
   e^{-\int_y^z \frac{\Lambda(w) - 1}{w - v \cdot G} \dw}
   \frac{\Lambda(z)}{z - v \cdot G} \dv \,,
\end{align*}
and the kernel of the adjoint operator $K^\ast$ is
\begin{align*}
   k^\ast (y, z) 
= \frac{1}{2} \One_{y < z} \int_\VV \One_{v \cdot G > z}
   e^{-\int_y^z \frac{\Lambda(w) - 1}{v \cdot G - w} \dw}
   \frac{\Lambda(y)}{v \cdot G - y} \dv
   + \frac{1}{2} \One_{y > z} \int_\VV \One_{v \cdot G < z}
   e^{-\int_z^y \frac{\Lambda(w) - 1}{w - v \cdot G} \dw}
   \frac{\Lambda(y)}{y - v \cdot G} \dv  \,.
\end{align*}
In order to show that $K$ is compact on $L^2(-|G|, \, |G|)$ we prove that $k \in L^2(\dy\dz)$ so that $K$ is a Hilbert-Schmidt operator. We only present the details for the $L^2$-bound of the first term in $k(y, v)$, for the bound of the second term follows in a similar way. 
Hence, by~\eqref{bound:expo-1} and the upper/lower bounds of $\Lambda$, we have
\begin{align*}
   \frac{1}{2} \One_{z < y} 
   \int_\VV \One_{v \cdot G > y}
   e^{-\int_z^y \frac{\Lambda(w) - 1}{v \cdot G - w} \dw}
   \frac{\Lambda(z)}{v \cdot G - z} \dv
&\leq
   \frac{1}{2} \One_{z < y}  
   \lambda_2 \int_\VV \One_{v \cdot G > y}
    \vpran{\frac{v \cdot G - y}{v \cdot G - z}}^{\lambda_1} 
    \frac{1}{v \cdot G - y} \dv 
\\
&\leq  
   \lambda_2 c_d \One_{z < y} 
   \int_{y/|G|}^1 
    \frac{\vpran{w|G| - y}^{\lambda_1 - 1}}
           {\vpran{w|G| - z}^{\lambda_1}} \dw \,.
\\
& \leq
   \lambda_2 c_d \One_{z < y}
   \vpran{\frac{1}{\lambda_1 |G|} + \frac{1}{|G|} \ln \frac{|G| - z}{y - z}}
\in L^2(\dy\dz) \,,
\end{align*}
where the last step follows from~\eqref{est:integral}. Hence $K$ is compact on $L^2(-|G|, \, |G|)$.

\smallskip

Next, we show that $R \in \vpran{\Null (\Id - K^\ast)}^{\perp}$. The exponential term in $k^\ast$ can be simplified as
\begin{align*}
    e^{-\int_y^z \frac{\Lambda(w) - 1}{v \cdot G - w} \dw}
= e^{-\int_y^z \frac{\Lambda(w)}{v \cdot G - w} \dw}
   e^{\int_y^z \frac{1}{v \cdot G - w} \dw}
= e^{-\int_y^z \frac{\Lambda(w)}{v \cdot G - w} \dw}
   \frac{v \cdot G - y}{v \cdot G - z}
\end{align*}
and similarly,
\begin{align*}
   e^{-\int_z^y \frac{\Lambda(w) - 1}{w - v \cdot G} \dw}
= e^{-\int_z^y \frac{\Lambda(w)}{w - v \cdot G} \dw}
   \frac{y - v \cdot G}{z - v \cdot G} \,.
\end{align*}
Hence, $k^\ast$ becomes
\begin{align*}
   k^\ast(y, z)
= \frac{1}{2} \One_{y < z} \int_\VV \One_{v \cdot G > z}
   e^{-\int_y^z \frac{\Lambda(w)}{v \cdot G - w} \dw}
   \frac{\Lambda(y)}{v \cdot G - z} \dv
   + \frac{1}{2} \One_{y > z} \int_\VV \One_{v \cdot G < z}
   e^{-\int_z^y \frac{\Lambda(w)}{w - v \cdot G} \dw}
   \frac{\Lambda(y)}{z - v \cdot G} \dv \,.
\end{align*}
The structure of the null space of $\Id - K^\ast$ can be seen by working with a modified kernel
\begin{align} \label{cond:normalization-k-ast}
    \tilde k^\ast(y, z)
= \frac{1}{\Lambda(y)} k^\ast (y, z) \Lambda(z) \,.
\end{align}
We claim that for any $y \in (-|G|, |G|)$, $\tilde k^\ast$ satisfies the normalization such that
\begin{align*}
  \int_{-|G|}^{|G|} \tilde k^\ast(y, z) \dz = 1 \,.
\end{align*}
Indeed, by using the spherical coordinates, the first term in $\int_{-|G|}^{|G|} \tilde k^\ast(y, z) \dz$ satisfies 
\begin{align*}
& \quad \,
   \frac{1}{2} \int_{-|G|}^{|G|} \One_{y < z} \int_\VV \One_{v \cdot G > z}
   e^{-\int_y^z \frac{\Lambda(w)}{v \cdot G - w} \dw}
   \frac{\Lambda(z)}{v \cdot G - z} \dv \dz
\\
&= \frac{1}{2} \int_{y}^{|G|} \int_{\Ss^{d-1}}
    \int_{z/|G|}^1 e^{-\int_y^z \frac{\Lambda(w)}{\mu |G| - w} \dw}
   \frac{\Lambda(z)}{\mu |G| - z} J(\mu, z) \dmu {\rm d}\hat v \dz
\\
&= \frac{1}{2} \int_{\Ss^{d-1}}\int_{y}^{|G|} 
    \int_{z/|G|}^1 e^{-\int_y^z \frac{\Lambda(w)}{\mu |G| - w} \dw}
   \frac{\Lambda(z)}{\mu |G| - z} J(\mu, z) \dmu \dz {\rm d}\hat v 
\\
& = \frac{1}{2} \int_{\Ss^{d-1}} \int_{y/|G|}^1 \int_{y}^{\mu |G|}
      e^{-\int_y^z \frac{\Lambda(w)}{\mu |G| - w} \dw}
   \frac{\Lambda(z)}{\mu |G| - z} J(\mu, z) \dz \dmu {\rm d}\hat v \,,
\end{align*} 
where $J$ is the Jacobian in the spherical coordinates. The last inner integral is evaluated as
\begin{align*}
   \int_{y}^{\mu |G|}
      e^{-\int_y^z \frac{\Lambda(w)}{\mu |G| - w} \dw}
   \frac{\Lambda(z)}{\mu |G| - z} \dz
&= \lim_{A \to \mu |G|}
   \int_{y}^{A}
      e^{-\int_y^z \frac{\Lambda(w)}{\mu |G| - w} \dw}
   \frac{\Lambda(z)}{\mu |G| - z} \dz
\\
& = 1 - \lim_{A \to \mu |G|} e^{-\int_y^A \frac{\Lambda(w)}{\mu |G| - w} \dw}
  = 1 \,.
\end{align*}
Hence, the first term in $\int_{-|G|}^{|G|} \tilde k^\ast(y, z) \dz$ is reduced to
\begin{align*}
      \frac{1}{2} \int_{-|G|}^{|G|} \One_{y < z} \int_\VV \One_{v \cdot G > z}
   e^{-\int_y^z \frac{\Lambda(w)}{v \cdot G - w} \dw}
   \frac{\Lambda(z)}{v \cdot G - z} \dv \dz
= \frac{1}{2} \int_{\Ss^{d-1}} \int_{y/|G|}^1 J(\mu, z) \dmu {\rm d}\hat v \,.
\end{align*}
Similarly, the second integral in $\int_{-|G|}^{|G|} \tilde k^\ast(y, z) \dz$ can be simplified as
\begin{align*}
   \frac{1}{2} \int_{-|G|}^{|G|} \One_{y > z} \int_\VV \One_{v \cdot G < z}
   e^{-\int_z^y \frac{\Lambda(w)}{w - v \cdot G} \dw}
   \frac{\Lambda(y)}{z - v \cdot G} \dv \dz
= \frac{1}{2} \int_{\Ss^{d-1}} \int^{y/|G|}_{-1} J(\mu, z) \dmu {\rm d}\hat v \,.
\end{align*}
Summing up gives the normalization condition~\eqref{cond:normalization-k-ast} for $\tilde k^\ast$. Therefore, if we denote $\tilde K^\ast$ as the operator with kernel $\tilde k^\ast$, then the normalization condition guarantees that the only elements in $\Null (\Id - \tilde K^\ast)$ are constants, which further implies that 
\begin{align*}
   \Null (\Id - K^\ast) = \Span \{\Lambda(y)\} \,.
\end{align*}
Hence, we have the equivalence such that
\begin{align} \label{cond:orthog-R}
   R \in \vpran{\Null (\Id - K^\ast)}^{\perp}
\Longleftrightarrow
   \int_{-|G|}^{|G|} R(y) \Lambda(y) \dy = 0 \,.
\end{align}
Now we check such orthogonality condition indeed holds. This can be done through a direct computation using the definition of $R$ given in~\eqref{def:R}. We can also show it by observing that $R = \vint{w}$ where $w$ is the solution to 
\begin{align} \label{eq:w}
   \del_y \vpran{(v \cdot G - y) w} = -\Lambda(y) w + R_1 \,,
\qquad
  w(\pm |G|, v) = 0 \,.
\end{align}
Then~\eqref{cond:orthog-R} holds by integrating~\eqref{eq:w} in $y$ and using condition~\eqref{cond:R-1} for $R_1$. 

Combining the orthogonality and the Fredholm property of $\Id - K$, we deduce that~\eqref{eq:vint-h} has a solution in $L^2(-|G|, \, |G|) \cap \vpran{\Null(I - K)}^{\perp}$. Denote this solution as $H_1$ and define
\begin{align*} 
   h_1(y, v)
= \begin{cases}
      \int_{-|G|}^y 
   e^{-\int_z^y \frac{\Lambda(w) - 1}{v \cdot G - w} \dw}
   \frac{1}{v \cdot G - z} 
   \vpran{\Lambda(z) H_1(z) + R_1(z, v)} \dz 
\qquad
   v \cdot G - y > 0 \,, \\[3pt]
      \int_y^{|G|}
   e^{-\int_y^z \frac{\Lambda(w) - 1}{w - v \cdot G} \dw}
   \frac{1}{z - v \cdot G} 
   \vpran{\Lambda(z) H_1(z) + R_1(z, v)}\dz 
\qquad
   v \cdot G - y < 0 \,.
   \end{cases}
\end{align*}
Then 
\begin{align*}
   \vint{h_1}
&= \frac{1}{2} \int_{\VV} \int_{-G}^y \One_{v \cdot G > y}
   e^{-\int_z^y \frac{\Lambda(w) - 1}{v \cdot G - w} \dw}
   \frac{1}{v \cdot G - z} 
   \vpran{\Lambda(z)H_1(z) + R_1(z, v)} \dz
\\
& \quad \, 
   + \frac{1}{2} \int_{\VV} \int_y^{|G|} \One_{v \cdot G < y}
   e^{-\int_y^z \frac{\Lambda(w) - 1}{w - v \cdot G} \dw}
   \frac{1}{z - v \cdot G} 
   \vpran{\Lambda(z)H_1(z) + R_1(z, v)}\dz
\\
&= H_1 \,,
\end{align*}
which shows $h_1$ is the unique solution to~\eqref{eq:h-general} satisfying~\eqref{cond:h-orthog} with $\vint{h_1} \in L^2(-|G|, \, |G|)$.
\end{proof}

\bigskip

\Ni{\bf Acknowledgements:} BP and WS would like to express their gratitude to the  NSF research network grant RNMS11-07444 (KI-Net). WS would also like to thank Laboratoire Jacques-Louis Lions (LJLL) at Sorbonne University for their hospitality and Centre National de la Recherche Scientifique (CNRS) for providing generous support.

BP has received funding from the European Research Council (ERC) under the European Union's Horizon 2020 research and innovation programme (grant agreement No 740623). The research of WS~is supported in part by NSERC Discovery Individual Grant R611626. MT is partially supported by NSFC 11871340. SY has received funding from the Japan Society for the Promotion of Science (JSPS) under the KAKENHI grant No 16K17554.

\bibliographystyle{amsxport}

\bibliography{InternalBP}



\end{document}